\documentclass[12pt]{article}   
\usepackage{graphicx,psfrag}
\usepackage{color}
\usepackage{array}
\usepackage{colortbl}
\usepackage{tabularx}
\newtheorem{theorem}{Theorem}
\newtheorem{corollary}[theorem]{Corollary}

\newtheorem{lemma}[theorem]{Lemma}

\newtheorem{example}[theorem]{Example}
\newenvironment {proof} {{\it
Proof.}}{\hspace*{\fill}$\Box$\par\vspace{4mm}}
\usepackage{amsmath}
\usepackage{amsfonts}
\newfont{\bb}{msbm10}

\def\:{\! :\!}

\usepackage{tikz}
\usetikzlibrary{arrows}
\usetikzlibrary{decorations.markings}

\tikzset{->-/.style={decoration={
  markings,
  mark=at position .5 with {\arrow[scale=1.5]{>}}},postaction={decorate}}}
\tikzset{-->-/.style={decoration={
  markings,
  mark=at position .8 with {\arrow[scale=1.5]{>}}},postaction={decorate}}}
\tikzset{->--/.style={decoration={
  markings,
  mark=at position .3 with {\arrow[scale=1.5]{>}}},postaction={decorate}}}

\usepackage{tikz}
\usetikzlibrary{decorations.markings}

\begin{document}

\title{Loopy, Hankel,  and Combinatorially Skew-Hankel Tournaments}

 \author{Richard A. Brualdi\\
 Department of Mathematics\\
 University of Wisconsin\\
 Madison, WI 53706\\
 {\tt brualdi@math.wisc.edu}
 \and
 Eliseu Fristcher\footnote{This research was part of this author's doctoral studies  performed while  visiting the University of Wisconsin - Madison and supported by a CAPES/BR Grant 0219/13-4.}  \\
 Instituto de Matem\'atica\\
 Universidade Federal Do Rio Grande Do Sul\\
 91509-900 Porto Alegre-RS, Brazil\\
 {\tt eliseu.fritscher@ufrgs.br}
 }

\maketitle

 \begin{abstract}  We investigate tournaments with a specified score vector having additional structure: loopy tournaments in which loops are allowed, Hankel tournaments which are tournaments symmetric about the Hankel diagonal (the anti-diagonal), and combinatorially skew-Hankel tournaments which are skew-symmetric about the Hankel diagonal. In each case, we obtain necessary and sufficient conditions for existence, algorithms for construction, and switches which allow one to move from any tournament of its type to any other, always staying
 within the defined type.


\end{abstract}

\section{Introduction}
Let $K_n$ be the complete graph with vertices $\{1,2,\ldots,n\}$.
A {\it tournament}  $T$ of order $n$ is a digraph obtained by orienting the edges of $K_n$.  Listing the vertices of $T$ in an arbitrary  order, we can form  the adjacency matrix of $T$, which is then called a {\it tournament matrix}. This is an $n\times n$ $(0,1)$-matrix
$T=[t_{ij}]$ such that $t_{ii}=0$ for all $i$ and $t_{ij}=1-t_{ji}$ for $i\ne j$, equivalently, such that $T+T^t=J_n-I_n$ where $T^t$ is the transpose of $T$ and $J_n$ is the $n\times n$ matrix of all 1s.   In general, we shall not distinguish between a tournament and a corresponding  tournament matrix. We refer to both as tournaments and label both as $T$ as we have done above. 

We define an $n\times n$  $(0,1)$-matrix $A=[a_{ij}]$ to be  {\it combinatorially skew-symmetric } about the main diagonal provided  that $a_{ij}=1-a_{ji}$ for $i\ne j$. Thus a tournament  matrix is combinatorially skew-symmetric about the main diagonal. An $n\times n$ tournament can be considered as the result of a round-robin tournament with players $1,2,\ldots,n$ where each pair of players compete in a game with player $i$ winning (respectively, losing) the game with player $j$ if and only if $t_{ij}=1$ (respectively, $t_{ij}=0$).     The {\it score vector} of $T$ is $R=(r_1,r_2,\ldots,r_n)$ where $r_i$ is the number of  wins of player $i$ and so the $i$th row sum of $T$. The score vector of $T$ is also the vector of outdegrees of the vertices of $T$. The score vector $R$ determines  the {\it losing vector} $S=(s_1,s_2,\ldots,s_n)$ where $s_i=(n-1)-r_i$ is the number of losses of player $i$  and so the $i$th column sum of $T$. This is because each player plays $n-1$ games  and the sum of the outdegree and indegree of each vertex is $n-1$.
Changing the order of the vertices replaces $T$ with $PTP^t$ for some permutation matrix $P$ whose score vector is $RP^t$.

A tournament matrix has $0$s on its main diagonal. This is natural from the competition point of view above. Yet, from the matrix point of view there is no reason to insist that only 0s occur on the main diagonal.  Thus we shall also consider $n\times  n$ tournaments $T=[t_{ij}]$ where each $t_{ii}$ may be 0 or 1. One possible interpretation of this is the following. Before the round-robin competition begins, each player $i$ flips a coin after calling heads or tails. If player $i$ calls the coin correctly, the player gets a point, and we set $t_{ii}=1$; otherwise, we set $t_{ii}=0$. A correct call adds 1 to a player's score. The score vector $R=(r_1,r_2,\ldots, r_n)$ is obtained, as above, by counting the number of 1s in each row of $T$. The score vector still determines the losing vector $S=(s_1,s_2,\ldots,s_n)$ since $r_i+s_i=n$ for all $i$. But now $S$ is not in general the vector of column sums of $T$.  If $S'=(s_1',s_2',\ldots,s_n')$  is the column sum vector of $T$, then $r_i+s_i'=n-1\mbox{ or } n+1$ depending on whether $t_{ii}=0$ or $t_{ii}=1$. We call a tournament $T$  obtained in this way a  {\it loopy tournament}, since viewed as a digraph there may be loops at the vertices. The trace of $T$ counts the number of correct calls of the coin flips.  Every tournament is also a loopy tournament. The  set of all  tournaments with score vector $R$  is usually denoted by ${\mathcal T}(R)$.  We denote the  set of all loopy tournaments with score vector $R$ by 
${\mathcal T}^{\ell}(R)$.

\begin{example}{\rm 
The following is a $5\times 5$ loopy tournament:
\[T=\left[\begin{array}{ccccc}
0&1&1&0&0\\
0&1&1&0&1\\
0&0&0&1&1\\
1&1&0&0&1\\
1&0&0&0&1\end{array}\right],\mbox{ where } T+T^t=\left[\begin{array}{ccccc}
0&1&1&1&1\\
1&1&1&1&1\\
1&1&0&1&1\\
1&1&1&0&1\\
1&1&1&1&1\end{array}\right].\]
The score vector of $T$ is $(2,3,2,3,2)$; the losing vector is $(3,2,3,2,3)$. The column sum vector of $T$ is $(2,3,2,1,4)$.
}\end{example}

We shall also consider tournaments with  additional structure. Tournaments and loopy tournaments are 
$n\times n$ $(0,1)$-matrices which are combinatorially skew-symmetric about the main diagonal. The {\it anti-diagonal} of an $n\times n$ matrix consists of the positions $\{(i,n+1-i): 1\le i\le n\}$. In a Hankel matrix, the entries in these positions are constant, as are the entries in each of  the other $2(n-1)$ diagonals parallel to the anti-diagonal. We use this association to refer to the anti-diagonal as the {\it Hankel diagonal}.   If $A=[a_{ij}]$ is an $n\times n$ matrix, then  we define its {\it Hankel transpose} to be the matrix $A^h=[a_{ij}']$ obtained from $A$ by transposing across the Hankel diagonal and thus for which $a_{ij}'=a_{n+1-j,n+1-i}$ for all $i$ and $j$.\footnote{By analogy with Toeplitz matrices, it would be natural to call the main diagonal the {\it Toeplitz diagonal} and to call  ordinary transpose the {\it Toeplitz transpose}.}   It is straightforward to check that, as for ordinary transpose, the Hankel transpose satisfies 
$(XY)^h=Y^h X^h$ for $n\times n$ matrices $X$ and $Y$, and that $(X^t)^h
=(X^{h})^t$, written as $X^{th}=X^{ht}$. 

A  {\it  Hankel tournament} is defined to be a tournament $T$ for which $T^h=T$. Thus a $(0,1)$-matrix $T=[t_{ij}]$ is a Hankel tournament if and only if
\[t_{n+1-j,n+1-i}=t_{ij}=1-t_{ji}=1-t_{n+1-i,n+1-j} \mbox{ for all $i\ne j$}.\]
 The entries on the Hankel diagonal of a Hankel tournament  can be 0 or 1 but,  by the combinatorial skew-symmetry of a tournament,  there must be  $\lfloor \frac{n}{2}\rfloor$ 1s on the Hankel diagonal; 
 if $n$ is odd, then the entry $t_{(n+1)/2,(n+1)/2}$ on the Hankel diagonal equals 0, since $T$ is a tournament. Reordering the columns of a Hankel tournament from last to first (and leaving the rows as is), we obtain a symmetric matrix $\widetilde{T}$  which has  $\lfloor \frac{n}{2}\rfloor$ 1s on the main diagonal, and thus $\widetilde{T}$ is the adjacency matrix of a loopy graph of order $n$ with $\lfloor \frac{n}{2}\rfloor$ loops. We call this graph the {\it Hankel loopy graph} of a Hankel tournament $T$ and denote it by $H(T)$.
 The set of all Hankel tournaments with score vector $R$ is denoted  by 
${\mathcal T}_H(R)$.


\begin{example}{\rm The following is a Hankel tournament $T$ of order 5 and its associated symmetric matrix $\widetilde{T}$:
\[T=\left[\begin{array}{c|c|c|c|c}
0&0&0&1&1\\ \hline
1&0&0&0&1\\ \hline
1&1&0&0&0\\ \hline
0&1&1&0&0\\ \hline
0&0&1&1&0\end{array}\right]\quad\mbox{and}\quad
\widetilde{T}=\left[\begin{array}{c|c|c|c|c}
1&1&0&0&0\\ \hline
1&0&0&0&1\\ \hline
0&0&0&1&1\\ \hline
0&0&1&1&0\\ \hline
0&1&1&0&0\end{array}\right]
\]
The graph $H(T)$ is a path $1-2-5-3-4$ with a loop at vertex 1 and at vertex 4.
}\end{example}

We also define a {\it combinatorially skew-Hankel tournament} to be an $n\times n$ $(0,1)$-matrix which  is combinatorially skew-symmetric about both the main diagonal  and the Hankel diagonal, and which has only 0s on both its main diagonal and its Hankel diagonal.
Let $D_n$ be the $n\times n$ $(0,1)$-matrix with 1s on the main diagonal  and on the Hankel diagonal and 0s elsewhere.
Thus  the $n\times n$ $(0,1)$-matrix $T=[t_{ij}]$ is a combinatorially skew-Hankel tournament if and only if
$T^t=J_n-D_n-T$ and $T^h=J_n-D_n-T$.  Thus, if $T$ is a combinatorially skew-Hankel tournament, then
$T^t=T^h$ or, put another way,  $T^{th}=T$.
In terms of its entries, $T$ is a  combinatorially skew-Hankel tournament if and only if
\begin{eqnarray*}
t_{ii}&=&t_{i,n+1-i}=0\mbox{ for all $i$,}\\
t_{ji}&=&1-t_{ij} \mbox{ for all $j\ne i,n+1-i$,}\\
t_{n+1-j,n+1-i} =1-t_{ij} \mbox{ for all $j\ne i,n+1-i$}.\\
\end{eqnarray*}
 Thus if $T$ is a  combinatorially skew-Hankel tournament, then  $t_{ij}=t_{n+1-i,n+1-j}$
 for all $i$ and $j$.
 Strictly  speaking, a combinatorially skew-Hankel tournament is not a tournament because skew-symmetry does not hold for symmetrically opposite elements on the Hankel diagonal. Viewing $T$ with respect to either its main diagonal or its Hankel diagonal, we have a round-robin tournament in which for each $i$, players $i$ and $n+1-i$  do not play a game. The set of all combinatorially skew-Hankel tournaments with a prescribed score vector is denoted by ${\mathcal T}_{H^*}(R)$.

\begin{example}{\rm The following matrix is a $5\times 5$ combinatorially skew-Hankel matrix:
\[\left[\begin{array}{c|c|c|c|c}
0&\cellcolor[gray]{0.8}1&\cellcolor[gray]{0.8}0&\cellcolor[gray]{0.8}0&0\\ \hline
0&0&\cellcolor[gray]{0.8}1&0&1\\ \hline
1&0&0&0&1\\ \hline
1&0&1&0&0\\ \hline
0&0&0&1&0\end{array}\right]\]
where  the four shaded entries determine all the  other entries off the main and Hankel diagonals.
	Note that, in general, a combinatorially  skew-Hankel matrix is invariant under a rotation by 180 degrees.}
\end{example}

In Section 2  we first extend Landau's theorem for the existence of a tournament with a prescribed score vector by weakening the usual  monotonicity assumption. We then use this theorem  to obtain necessary and sufficient conditions for the existence of a loopy tournament with a prescribed score vector. We show that this result can also be obtained by associating with an $n\times n$ loopy tournament an $(n+1)\times (n+1)$ ordinary tournament and then using Landau's theorem directly. In Section 3 we consider Hankel tournaments and  obtain necessary and sufficient conditions for the existence of a Hankel tournament with a prescribed score vector. 
In Section 4 we consider combinatorially skew-Hankel tournaments and  obtain necessary and sufficient conditions for their existence  with a prescribed score vector. 
We also present algorithms for construction of these three  types of tournaments. 

Let $T_1$ and $T_2$ be two tournaments in ${\mathcal T}(R)$ considered as digraphs.
It is a basis fact (see e.g. \cite{Br,Ry}) that one may go  from $T_1$ to $T_2$ by a finite sequence of  switches each of which reverses the directions of the edges of a 3-cycle (so resulting in another tournament in ${\mathcal T}(R)$). We identify  elementary moves that enable one to go from $T_1$ to $T_2$ when (i) $T_1,T_2\in {\mathcal T}^{\ell}(R)$, (ii) $T_1,T_2\in {\mathcal T}_H(R)$, and (iii) $T_1,T_2\in {\mathcal T}_{H^*}(R)$. In each case, the elementary moves always produce another tournament in the same class.

\section{Loopy Tournaments}

One of the most well-known theorems on  tournaments is Landau's theorem \cite{La} (see also \cite{Br,BF,Ma,Re,Ry,Th}) which asserts
that a vector $R=(r_1,r_2,\ldots,r_n)$ of nonnegative integers is the score vector of  an $n\times n$  tournament  if and only if
\begin{equation}\label{eq:landau}
\sum_{i\in J} r_i\ge {{|J|}\choose 2},\ (J \subseteq \{1,2,\ldots,n\}), \mbox{ with equality if $J=\{1,2,\ldots,n\}$}.
\end{equation}
If $T$ is a tournament with score vector $R$, then for each permutation matrix $P$, $PTP^t$ is a tournament with score vector $RP^t$. Thus without loss of generality, one can assume that $R$ is monotone nondecreasing, that is, $r_1\le r_2\le\cdots\le r_n$. With this assumption,  (\ref{eq:landau}) is equivalent to
\begin{equation}\label{eq:landau2}
\sum_{i=1}^k r_i\ge {k\choose {2}},\ (k=1,2,\ldots,n), \mbox{ with equality if $k=n$}.
\end{equation}
It is usually under the  assumption that $R$ is nondecreasing that Landau's theorem is proved.
But this  monotone assumption can be weakened to provide  a somewhat stronger theorem.  Let $k$ be a nonnegative integer. The vector $R$ is {\it $k$-nearly nondecreasing} provided there is a $(0,1,\ldots,k)$-vector $u=(u_1,u_2,\ldots,u_n)$  such that
$R-u$ is nondecreasing. Thus $R$ is $k$-nearly  nondecreasing if and only if  $r_j\ge r_i-k$ for $1\le i<j\le n$.  If $k=0$, then $R$ is nondecreasing. If $k=1$, then we use {\it nearly nondecreasing} (as used in \cite{TSM})  instead of  $k$-nearly nondecreasing.
For example, $(3,2,3,4,3,4)$
is nearly nondecreasing.

\begin{lemma}\label{lem:nearly} Let $R=(r_1,r_2,\ldots,r_n)$ 
 be a $2$-nearly nondecreasing  vector of nonnegative integers. Assume that  
   $R$  satisfies Landau's inequalities $(\ref{eq:landau2})$. Then   the nondecreasing rearrangement of $R$
also satisfies Landau's inequalities $(\ref{eq:landau2})$.
 \end{lemma}
 
 \begin{proof} We prove the lemma by induction on the number $\alpha$ of pairs $(i,j)$ such that $i<j$ and $r_i-2\le r_j\le r_i-1$. If $\alpha=0$, then $R$ is nondecreasing and there is nothing to prove. Assume that $R$ is not nondecreasing, and suppose that $i$ and $j$ satisfy $i<j$ and $r_j\le r_i-1$. Then for each $k$ with  $i<k<j$, $r_k\ge r_i-2$. Hence there exists $p$ with $i\le p<j$ such that
 $r_{p+1}=r_p-l$ where $l=1\mbox{ or } 2$. We switch $r_p$ and $r_{p+1}$ thereby decreasing $\alpha$. Suppose that
 $(\sum_{i=1}^{p-1}r_i)+r_{p+1}< {p\choose 2}$. Since $\sum_{i=1}^pr_i\ge {p\choose 2}$,
 we conclude that $\sum_{i=1}^pr_i={p\choose 2}+h$ where $0\le h\le l-1$.
 Since $R$ satisfies Landau's inequalities (\ref{eq:landau2}), we calculate that
 \[r_{p+1}=\left(\sum_{i=1}^{p+1}r_i\right)-\left(\sum_{i=1}^p r_i\right)\ge {{p+1}\choose 2}-\left({p\choose 2}+h\right)=p-h.\]
 and 
 \[r_p =\left(\sum_{i=1}^pr_i\right) -\left(\sum_{i=1}^{p-1}r_i\right)\le \left({p\choose 2}+h\right)-{{p-1}\choose 2} =p-1+h.\]
 Thus $p-h+l \le r_{p+1}+l=r_p\le  p-1+h$, and  so $l\le 2h-1$. Since $h\le l-1$, we have $l\ge 3$, a contradiction.
 We conclude that  $(\sum_{i=1}^{p-1}r_i)+r_{p+1}\ge  {p\choose 2}$, and thus that  switching the order of $r_p$ and $r_{p+1}$ produces a vector  $R'$ that also satisfies Landau's inequalities. Since $\alpha$ is decreased, the lemma follows by induction.
 \end{proof}

\begin{theorem}\label{th:newlandau}
Let $R=(r_1,r_2,\ldots,r_n)$ be a $2$-nearly nondecreasing  vector of nonnegative integers. Then $R$ is the score vector of a tournament if and only if $(\ref{eq:landau2})$ holds.
\end{theorem}

\begin{proof} If $R$ is the score vector of a tournament, then (\ref{eq:landau2}) holds. Now assume that (\ref{eq:landau2}) holds. By Lemma \ref{lem:nearly} the nondecreasing rearrangement $R'$ of $R$ satisfies (\ref{eq:landau2}). Hence by Landau's theorem, there is a tournament $T$ with score vector $R$.
There is an $n\times n$ permutation matrix $P$ such that 
$R'=RP^t$. Then $PTP^t$ is a tournament with score vector $R$.\end{proof}

Let $T=[t_{ij}]$ be a loopy tournament with score vector $R=(r_1,r_2,\ldots,r_n)$. Since simultaneously permuting the rows and columns of $T$ results in another loopy tournament (with the same number of loops), we assume without loss of generality that $R$ satisfies $r_1\le r_2\le \cdots\le r_n$. The number of 1s on the main diagonal of $T$ is some integer, denoted as  $n-t$,  where  $0\le t\le n$.
The sum of all the entries of $T$ equals ${n\choose 2}+(n-t)$. Landau's theorem for the existence of a tournament with a prescribed score vector can be used to determine the existence of a loopy tournament with a prescribed score vector. If $a$ is an integer, then $a^+=\max\{0,a\}$.

\begin{theorem}\label{th:sm}
Let $R=(r_1,r_2,\ldots,r_n)$ be a vector of nonnegative integers with $r_1\le r_2\le\cdots\le r_n$.
Then there exists an loopy tournament with score vector $R$ if and only if there is an integer $t$ with $0\le t\le n$ such that
\begin{equation}
\label{eq:sm}
\sum_{i=1}^kr_i\ge {k\choose 2}+(k-t)^+ , (k=1,2,\ldots,n),  \mbox{ with equality when $k=n$}.\end{equation}
When these conditions are satisfied, the number of $1${\rm s} on the main diagonal of the loopy tournament is $n-t$ and these can be taken to be in the last $(n-t)$ positions on the main diagonal.
\end{theorem}

\begin{proof} The condition (\ref{eq:sm}) is clearly necessary for the existence of a loopy tournament with score vector $R$ and $n-t$ loops.
Now suppose that (\ref{eq:sm}) holds.
We show that there is a loopy tournament with score vector $R$ for which the $(n-t)$ 1s on the main diagonal occur in those positions corresponding to the $(n-t)$ largest $r_i$. Let $R'=(r_1',r_2',\ldots,r_n')$ be obtained from $R$ by subtracting 1 
from $r_{t+1},r_{t+2},\ldots,r_n$. 
Then $R'$ is nearly  nondecreasing and (\ref{eq:sm}) implies that $\sum_{i=1}^k r_i'\ge {k\choose 2}$ for $1\le k\le n$ with equality for $k=n$. Hence by Theorem \ref{th:newlandau} there exists a tournament $T'$ with score vector $R'$. Replacing the 0s with 1s in the last $(n-t)$ positions on the main diagonal of $T'$ gives a loopy tournament with score vector $R$.\end{proof}


 Let $R$ be a nearly nondecreasing vector of nonnegative integers  satisfing
(\ref{eq:landau2}).   By Theorem \ref{th:newlandau} there exists a tournament $T$ with score vector $R$. Simultaneously permuting the rows and columns of $T$ gives a tournament with score vector $R'$ obtained by applying the same permutation to $R$. Thus to find $T$ we can first permute  $R$ to get a nondecreasing $R'$, apply a known algorithm (for instance, the algorithm of Ryser \cite{Ry,Br})
to obtain a tournament $T'$  with score vector $R'$, and then simultaneously permute the rows and columns of $T'$ to obtain $T$ with score vector $R$. 

To construct a loopy tournament with nondecreasing score vector $R$ satisfying (\ref{eq:sm}), we can first subtract 1 from the largest $n-t$ components of $R$, and this results in a nearly nondecreasing vector $R'$ which by Theorem \ref{th:newlandau} satisfies (\ref{eq:landau2}). Then we can construct a tournament $T'$ with score vector $R'$ as just outlined above.  Replacing the 0s   in the last $(n-t)$ positions on the main diagonal with 1s results in an loopy tournament with score vector $R$.

The proof of Theorem \ref{th:sm} uses the strengthened form of Landau's theorem given in Theorem \ref{th:newlandau} in order to show existence of a loopy tournament with a prescribed nondecreasing score vector. For a given nondecreasing vector $R$ of nonnegative integers, we now establish a bijection between  ${\mathcal T}^{\ell}(R)$ and ${\mathcal T}(R')$ for a certain vector $R'$. This enables us to identify basic switches,  a sequence of which allows us  to go from a $T_1\in{\mathcal T}^{\ell}(R)$ to a $T_2\in {\mathcal T}^{\ell}(R)$ with
all intermediate matrices also in ${\mathcal T}^{\ell}(R)$.

\begin{theorem}\label{th:newloopy}
Let $R=(r_1,r_2,\ldots,r_n)$ be a nondecreasing vector of nonnegative integers such that there is an integer $t$ with $0\le t\le n$ such that
\[\sum_{i=1}^nr_i={n\choose 2}+(n-t).\]
Let $R'=(t,r_1,r_2,\ldots,r_n)$. Then there is a bijection between 
${\mathcal T}^{\ell}(R)$ and ${\mathcal T}(R')$. In particular, 
${\mathcal T}^{\ell}(R)\ne\emptyset$ if and only if ${\mathcal T}(R')\ne\emptyset$. Moreover,  these sets are nonempty if and only if $(\ref{eq:sm})$ holds.
\end{theorem}

\begin{proof}
Let $T=[t_{ij}]$ be an $n\times n$  loopy tournament in ${\mathcal T}^{\ell}(R)$, and let an $(n+1)\times (n+1)$ matrix $T'=[t_{ij}']$ with rows and columns indexed by $0,1,2,\ldots,n$  be defined by
\[t_{ij}'=\left\{\begin{array}{cl}
t_{ij},&\mbox{if $1\le i,j\le n$ and $i\ne j$},\\
t_{ii},&\mbox{if  $j=0$ and $1\le i\le n$},\\
1-t_{jj},&\mbox{if $i=0$ and $1\le j\le n$},\\
0,&\mbox{if $0\le i=j\le n$.}\end{array}\right.\]
Thus $T'$ is obtained from $T$ by horizontally moving the entries on the main diagonal to column 0,  vertically moving 1 minus the entries on the main diagonal to row 0, and putting 0s everywhere on the main diagonal. Since $T$ has $(n-t)$ 1s on its main diagonal, $T'$ is a tournament in ${\mathcal T}(R')$. This mapping is reversible and hence 
there is a bijection between 
${\mathcal T}^{\ell}(R)$ and ${\mathcal T}(R')$.

As remarked in the proof of Theorem \ref{th:sm}, the inequalities (\ref{eq:sm}) are necessary for 
${\mathcal T}(R')$, and hence ${\mathcal T}^{\ell}(R)$, to be  nonempty.
 Now assume that   (\ref{eq:sm}) holds.
 Let $R''=(r_1'',r_2'',\ldots,r_n'',r_{n+1}'')$ be obtained from  $R'$ by reordering its entries  to be nondecreasing.
Then $R''=(r_1,\ldots, r_p,t,r_{p+1},\ldots,r_n)$ where $r_p\le t\le r_{p+1}$.  We have
\[\sum_{i=1}^kr_i''=
\left\{\begin{array}{cl}
\sum_{i=1}^kr_i\ge {k\choose 2}+(k-t)^+\ge {k\choose 2},&\mbox{ if $k\le p$},\\
&\\
(\sum_{i=1}^k r_i) +t\ge {k\choose 2}+t+(k-t)^+,&
\mbox{ if $p+1\le k\le n+1$.} \end{array}\right.\]
Since $t+(k-t)^+\ge k$, we have that 
\[ {k\choose 2}+t+(k-t)^+\ge {k\choose 2}+k={{k+1}\choose 2}.\]
Thus by Landau's theorem, ${\mathcal T}(R'')$ is nonempty, and hence 
${\mathcal T}(R')$ and ${\mathcal T}^{\ell}(R)$ are nonempty.
\end{proof}

To construct a loopy tournament in ${\mathcal T}^{\ell}(R)$, we can use any algorithm to construct a tournament in ${\mathcal T}(R')$, and then use the bijection in Theorem \ref{th:newloopy}.

As remarked in the introduction, given any two tournaments $T_1'$ and $T_2'$ with the same score vector $R'$, then by reversing 3-cycles we can get from $T_1'$ to $T_2'$ where  all intermediate matrices are tournaments with
score vector $R'$. The operation of reversing a 3-cycle $i\rightarrow j\rightarrow k\rightarrow i$  in terms of matrices is that of switching  the $3\times 3$  matrix $T'[i,j,k]$  determined by rows and columns $i,j,\mbox{ and }k$  of $T'$ as shown below:
\begin{equation}\label{eq:interchange}
\begin{array}{c||c|c|c}
&i&j&k\\ \hline \hline
i&x&1&0\\ \hline
j&0&y&1\\ \hline
k&1&0&z\end{array} \rightarrow
\begin{array}{c||c|c|c}
&i&j&k\\ \hline \hline
i&x&0&1\\ \hline
j&1&y&0\\ \hline
k&0&1&z\end{array}.
\end{equation}
We call this a {\it $3$-cycle switch} and denote it by $\bigtriangleup_{i,j,k}$.  A 3-cycle switch is reversible and its inverse is 
$\bigtriangleup_{k,j,i}$.

We now take $R'$ and $R$ as given in the proof of Theorem \ref{th:newloopy}, and
using the bijection between ${\mathcal T}^{\ell}(R)$ and ${\mathcal T}(R')$ given there,  we identify the  switches  in ${\mathcal T}^{\ell}(R)$ corresponding to the $3$-cycle switches in ${\mathcal T}(R')$.
There are two possibilities: the switch in ${\mathcal T}(R')$ is in rows and columns (a)  $\{i,j,k\}\subseteq\{1,2,\ldots,n\}$ or (b) $\{0,i,j\}$ where
$\{i,j\}\subset\{1,2,\ldots,n\}$.

If (a), then the switches for ${\mathcal T}(R')$ and  ${\mathcal T}^{\ell}(R)$ are the same and hence  
$\bigtriangleup_{i,j,k}$ is an operation that maintains the tournament in the class ${\mathcal T}^{\ell}(R)$.

If (b), then the  switch $\bigtriangleup_{0,i,j}$ for ${\mathcal T}(R')$ is
\[
\begin{array}{c||c|c|c}
&0&i&j\\ \hline\hline
0&0&1&0\\ \hline
i&0&0&1\\ \hline
j&1&0&0\end{array}\rightarrow
\begin{array}{c||c|c|c}
&0&i&j\\ \hline\hline
0&0&0&1\\ \hline
i&1&0&0\\ \hline
j&0&1&0\end{array},\]
which, by moving the 1
  in column $0$ of the left matrix to column $j$ 
and moving the 1 in column 0 of the right matrix to column $i$, becomes
\[\begin{array}{c||c|c}
&i&j\\ \hline\hline
i&0&1\\ \hline
j&0&1\end{array}
\rightarrow
\begin{array}{c||c|c}
&i&j\\ \hline\hline
i&1&0\\ \hline
j&1&0\end{array}.\]
In terms of the digraph, the operation (b)    reverses the direction of  an edge  from a non-loop vertex $i$ to a loop-vertex $j$ and moves the loop from $j$ to $i$.
We call this operation an  {\it edge-loop switch} and denote it by $\hbox{$\rightarrow$}\kern -1.5pt\hbox{$\circ$}_{ij}$. This operation is also reversible  and  the inverse of $\hbox{$\rightarrow$}\kern -1.5pt\hbox{$\circ$}_{ij}$ is $\hbox{$\rightarrow$}\kern -1.5pt\hbox{$\circ$}_{ji}$.
Thus we have the following theorem.

\begin{theorem}\label{loopyswitch} Let $T_1$ and $T_2$ be two  loopy tournaments in ${\mathcal T}^{\ell}(R)$. Then $T_1$ can be brought to $T_2$ by a sequence of $3$-cycle switches and  edge-loop switches where each switch produces a loopy tournament in ${\mathcal T}^{\ell}
(R)$.
\end{theorem}

To conclude this section, we remark that, any $n\times n$  tournament 
can always be collapsed to an $(n-1)\times (n-1)$  loopy tournament using any row and column $i$ with the 1s in column $i$ moved horizontally to the positions on the main diagonal, and then deleting row and column $i$. 

\section{Hankel Tournaments}

   We first characterize the score vectors of Hankel tournaments.
 Let $T$ be a Hankel tournament, and let $i$ be an integer with $1\le i\le n$. 
 Since $T$ is a tournament, row $i$ of $T$ determines column $i$, and since $T$ is a Hankel tournament, it also determines row and column $n+1-i$.
We illustrate this important property in the next example.

\begin{example}{\rm
Let $T$ be an $8\times 8$ Hankel tournament with row 3 given. 
Then the entries in column 3, and row and column $8+1-3=6$ are determined as shown in
\[\left[\begin{array}{c|c||c||c|c||c||c|c}
0&&1-a&&&\phantom{11}g\phantom{11}&&\cellcolor[gray]{0.8}\\ \hline
&0&1-b&&&f&\cellcolor[gray]{0.8}&\\ \hline\hline
a&b&0&c&d&\cellcolor[gray]{0.8}e&f&g\\ \hline\hline
&&1-c&0&\cellcolor[gray]{0.8}&d&&\\ \hline
&&1-d&\cellcolor[gray]{0.8}&0&c&&\\ \hline\hline
1-g&1-f&\cellcolor[gray]{0.8}1-e&1-d&1-c&0&1-b&1-a\\ \hline\hline
&\cellcolor[gray]{0.8}&1-f&&&b&0&\\ \hline
\cellcolor[gray]{0.8}&&1-g&&&a&&0\end{array}\right],\]
where the  $0$s on the main diagonal have been inserted and the Hankel diagonal has been shaded.}
\end{example}

Let  $R=(r_1,r_2,\ldots,r_n)$ be the score vector of $T$, and let $S=(s_1,s_2,\ldots,s_n)$ be the column sum vector of $T$.  We have $\sum_{i=1}^nr_i={n\choose 2}$.  Since $T$ is symmetric about the Hankel diagonal,  $r_i=s_{n+1-i}$. Since $T$ is a  tournament, 
$r_{n+1-i}=(n-1)-s_{n+1-i}=(n-1)-r_i$ and thus the score vector $R$ satisfies the {\it Hankel property}
\begin{equation}\label{eq:Hankelprop}
r_i+r_{n+1-i}=n-1 \mbox{ for all $i=1,2,\ldots,n$.}\end{equation}
Thus, if $n$ is odd, 
\[r_{(n+1)/2}=\frac{n-1}{2}.\]

We next  show that  there is no loss of generality in assuming that $R$ is nondecreasing.

\begin{lemma}\label{lem:monotone}
Let $R=(r_1,r_2,\ldots,r_n)$ be a vector of nonnegative integers, and let $R'=(r_1',r_2',\ldots,r_n')$  be obtained from $R$ by rearranging its entries so that
$r_1'\le r_2'\le\cdots\le r_n'$. Then there exists a Hankel tournament  $T$ with score vector $R$ if and only if there exists a Hankel tournament $T'$ with score vector $R'$. Moreover, $T'$ is obtained by  
simultaneously permuting the rows and columns of $T$. 
\end{lemma}

\begin{proof}
Let $T$ be a Hankel tournament with score vector $R$.
We first note that if for some $i$, we interchange rows $i$ and $(n+1-i)$
and simultaneously interchange columns $i$ and $(n+1-i)$ of $T$, then the resulting matrix is a Hankel tournament
with score vector obtained from $R$ by interchanging $r_i$ and $r_{n+1-i}$.
Thus using a sequence of pairwise interchanges of this sort, we may assume that $r_i\le r_{n+1-i}$ for $1\le i\le \lfloor n/2\rfloor$.  Since $r_i+r_{n+1-i}=n-1$, we have
\[r_i\le \frac{n-1}{2}\le r_{n+1-i},\mbox{ where, if $n$ is odd, $r_{(n+1)/2}=\frac{n-1}{2}$.}\]
If  for some $i$ and $j$ with $1\le i<j\le \lfloor n/2\rfloor$, we have
$r_i>r_j$, then we also have that $\lceil n/2\rceil\le n+1-j<n+1-i\le n$ and $r_{n+1-j}>r_{n+1-i}$.
Thus we may interchange rows  $i$ and $j$ and columns $i$ and $j$, and also interchange rows  $(n+1-i)$ and $(n+1-j)$ and columns $(n+1-i)$ and $(n+1-j)$, and obtain a Hankel tournament with row sum vector obtained from $R$ by interchanging $r_i$ and $r_j$ and by interchanging $r_{n+1-i}$ and $r_{n+1-j}$. 
Thus by a sequence of double pairwise interchanges of this sort, we are able to get a Hankel tournament whose row sum vector is a monotone rearrangement of $R$.
The converse follows in a similar way.
\end{proof}

Because of Lemma \ref{lem:monotone}, for the existence of Hankel tournaments  with score vector $R$, it suffices to assume that $R$ is nondecreasing. We show that the only condition needed in addition to Landau's inequalities is that $R$ satisfy the Hankel property.

\begin{theorem}\label{th:Hankel}
Let $R=(r_1,r_2,\ldots,r_n)$ be a vector of nonnegative integers with $r_1\le r_2\le \cdots\le r_n$.
Then there exists a Hankel tournament with score vector $R$ if and only if
\begin{equation}\label{eq:Hankel1}
r_i+r_{n+1-i}=n-1,\quad (i=1,2,\ldots,n) \end{equation}
and
\begin{equation}\label{eq:Hankel2}
\sum_{i=1}^k r_i\ge {k\choose {2}},\ (k=1,2,\ldots,n), \mbox{ with equality if $k=n$}.
\end{equation}
\end{theorem}

\begin{proof}
We know that (\ref{eq:Hankel1}) and (\ref{eq:Hankel2}) are necessary for the existence of a Hankel tournament with score vector $R$. 
Suppose, to the contrary, that the converse is false. Then clearly $n\ge 3$. We then choose an $R=(r_1,r_2,\ldots,r_n)$ satisfying  (\ref{eq:Hankel1}) and (\ref{eq:Hankel2})  for which a  Hankel tournament with score vector $R$ does not exist, where $n$ is minimum and for this $n$, $r_1$ is minimum.
We consider two cases.

\smallskip\noindent
{\it Case} $1$: For each $k$ with $1\le k\le n-1$, we have strict inequality in (\ref{eq:Hankel2}).

\smallskip
In this case we have $r_1\ge 1$ and we consider $R'=(r_1',r_2',\ldots,r_n')=(r_1-1,r_2,\ldots,r_{n-1},r_n+1)$. Then  with $r_i'$ replacing $r_i$, $R'$ satisfies the corresponding conditions (\ref{eq:Hankel1}) and (\ref{eq:Hankel2}). Thus by the minimality condition on $r_1$, there exists a Hankel tournament $T=[t_{ij}]$
with score vector $R'$. Since $r_n'-r_1'=(r_n+1)-(r_1-1)=(r_n-r_1)+2\ge 2$, there exists an integer $p$ with $2\le p\le n-1$, such that $t_{1p}=0$ and $t_{np}=1$.  Since $T$ is a tournament, we have
$t_{p1}=1$ and $t_{pn}=0$. Since $T$ is a Hankel tournament, we also have
$t_{n+1-p,n}=0$, $t_{n+1-p,1}=1$, $t_{n,n+1-p}=1$, and $t_{1,n+1-p}=0$. Thus with $q=n+1-p$, we have the structure:
\[\begin{array}{c||c|c|c|c|c|c|c}
&1&\cdots&p&\cdots&q&\cdots&n\\ \hline\hline
1&&&0&&0&&\\ \hline
\vdots&&&&&&&\\ \hline
p&1&&0&&a&&0\\ \hline
\vdots&&&&&&&\\ \hline
q&1&&1-a&&0&&0\\ \hline
\vdots&&&&&&&\\ \hline
n&&&1&&1&& \end{array}  
\] where $a$, on the Hankel diagonal,   is $0$ or $1$.
If $a=1$, then replacing as shown below:
\[
\begin{array}{c||c|c|c|c|c|c|c}
&1&\cdots&p&\cdots&q&\cdots&n\\ \hline\hline
1&&&0&&\cellcolor[gray]{0.8}0&&\\ \hline
\vdots&&&&&&&\\ \hline
p&1&&0&&\cellcolor[gray]{0.8}1&&\cellcolor[gray]{0.8}0\\ \hline
\vdots&&&&&&&\\ \hline
q&\cellcolor[gray]{0.8}1&&\cellcolor[gray]{0.8}0&&0&&0\\ \hline
\vdots&&&&&&&\\ \hline
n&&&\cellcolor[gray]{0.8}1&&1&& \end{array}\quad
\longrightarrow\quad
\begin{array}{c||c|c|c|c|c|c|c}
&1&\cdots&p&\cdots&q&\cdots&n\\ \hline\hline
1&&&0&&\cellcolor[gray]{0.8}1&&\\ \hline
\vdots&&&&&&&\\ \hline
p&1&&0&&\cellcolor[gray]{0.8}0&&\cellcolor[gray]{0.8}1\\ \hline
\vdots&&&&&&&\\ \hline
q&\cellcolor[gray]{0.8}0&&\cellcolor[gray]{0.8}1&&0&&0\\ \hline
\vdots&&&&&&&\\ \hline
n&&&\cellcolor[gray]{0.8}0&&1&& \end{array}\ ,\]
we obtain a Hankel tournament with score vector $R$. If $a=0$, then  replacing as shown below:
\[
\begin{array}{c||c|c|c|c|c|c|c}
&1&\cdots&p&\cdots&q&\cdots&n\\ \hline\hline
1&&&\cellcolor[gray]{0.8}0&&0&&\\ \hline
\vdots&&&&&&&\\ \hline
p&\cellcolor[gray]{0.8}1&&0&&\cellcolor[gray]{0.8}0&&0\\ \hline
\vdots&&&&&&&\\ \hline
q&1&&\cellcolor[gray]{0.8}1&&0&&\cellcolor[gray]{0.8}0\\ \hline
\vdots&&&&&&&\\ \hline
n&&&1&&\cellcolor[gray]{0.8}1&& \end{array}\quad\longrightarrow\quad
\begin{array}{c||c|c|c|c|c|c|c}
&1&\cdots&p&\cdots&q&\cdots&n\\ \hline\hline
1&&&\cellcolor[gray]{0.8}1&&0&&\\ \hline
\vdots&&&&&&&\\ \hline
p&\cellcolor[gray]{0.8}0&&0&&\cellcolor[gray]{0.8}1&&0\\ \hline
\vdots&&&&&&&\\ \hline
q&1&&\cellcolor[gray]{0.8}0&&0&&\cellcolor[gray]{0.8}1\\ \hline
\vdots&&&&&&&\\ \hline
n&&&1&&\cellcolor[gray]{0.8}0&&\end{array}\ ,\]
we obtain  a Hankel tournament with score vector $R$. 
Both possibilities give a contradiction. Note that if $p=q$, then their common value is $(n+1)/2$ and the number of replacements above is only four.

\smallskip\noindent
{\it Case} $2$: There is a $k$ with $1\le k\le n-1$ such that we have equality in (\ref{eq:Hankel2}), that is, $\sum_{i=1}^kr_i={k\choose 2}$.

\smallskip
We  calculate that
\begin{eqnarray*}
\sum_{i=1}^{n-k}r_i &=&{n\choose 2}-\sum_{i=n-k+1}^nr_i
={n\choose 2}-\sum_{i=1}^k(n-1-r_i)\\
&=&{n\choose 2}-k(n-1)+{k\choose 2}
={{n-k}\choose 2}.
\end{eqnarray*}
Hence we may assume that $k\le n/2$, and we also have that $\sum_{i=1}^{n-k}r_i ={{n-k}\choose 2}$.

We now  calculate that
\begin{eqnarray*}
r_{k+1}&=&\sum_{i=1}^{k+1}r_i-\sum_{i=1}^k r_i
= \sum_{i=1}^{k+1}r_i-{k\choose 2}\\
&\ge& {{k+1}\choose 2}-{k\choose 2}=k.\end{eqnarray*}
Thus by the monotonicity assumption, $r_i\ge k$ for $k+1\le i\le n$. Similarly, $r_k\le k-1$.

It follows from Landau's theorem that there exists a tournament $T_1$ with score vector
$R_1=(r_1,r_2,\ldots,r_k)$. Since $r_i+r_{n-i}=n-1$ for all $i$ and $r_k\le k-1$, we have
$r_{n-k+1}\ge (n-1)-(k-1)=n-k$. The monotonicity assumption on $R$ now implies
that $R_2=(r_{n-k+1}-(n-k),r_{n-k+2}-(n-k),\ldots,r_n-(n-k))$ is a vector of nonnegative integers.
Since for $1\le i\le k$, we have
\[r_i+(r_{n+1-i}-(n-k))=r_i+r_{n+1-i}-(n-k)=(n-1)-(n-k)=k-1,\]
then $s_i=r_{n+1-i}-(n-k)$ is the $i$th column sum of $T_1$.
Hence the row sum vector of $T_1^h$ equals $R_2$.

To summarize thus far, the matrix
\[T=\left[\begin{array}{c|c|c}
T_1&O_{k,n-2k}&O_{k,k}\\ \hline
J_{n-2k,k}&X&O_{n-2k,k}\\ \hline
J_{k,k}&J_{k,n-2k}&T_1^h
\end{array}\right]\]
will be a Hankel tournament  with score vector $R$, provided we can choose $X$ as an $(n-2k)\times (n-2k)$ Hankel tournament  with score vector $(r_{k+1}-k,r_{k+2}-k,\ldots,r_{n-k}-k)$, 
which we know is nonnegative.
We calculate that for $1\le l\le n-2k$,
\begin{eqnarray*}
 \sum_{j=1}^l (r_{k+j}-k) &=&\left(\sum_{i=1}^{k+l}r_i-\sum_{i=1}^kr_i\right)-lk
 =\sum_{i=1}^{k+l}r_i-{k\choose 2}-lk\\
 &\ge& {{k+l}\choose 2}-{k\choose 2}-lk
 ={l\choose 2}.
 \end{eqnarray*}
Since by assumption $r_{k+j}+r_{n+1-(k+j)}=n-1$, we have
\[(r_{k+j}-k)+(r_{n+1-(k+j)}-k)=(n-2k)-1.\]
The minimality assumption now implies that the Hankel tournament  $X$ exists, and hence we have a Hankel tournament  with score vector $R$. This contradiction completes this case, and thus the proof of the theorem is complete.
\end{proof}

As a corollary we show that the nondecreasing assumption in Theorem \ref{th:Hankel} can be weakened to 2-nearly nondecreasing.

\begin{corollary}\label{cor:Hankel}
Let $R=(r_1,r_2,\ldots,r_n)$ be a $2$-nearly nondecreasing vector of nonnegative integers.
Then there exists a Hankel tournament with score vector $R$ if and only if
\begin{equation}\label{eq:Hankel1a}
r_i+r_{n+1-i}=n-1,\quad (i=1,2,\ldots,n) \end{equation}
and
\begin{equation}\label{eq:Hankel2a}
\sum_{i=1}^k r_i\ge {k\choose {2}},\ (k=1,2,\ldots,n), \mbox{ with equality if $k=n$}.
\end{equation}
\end{corollary}

\begin{proof}
As before the conditions (\ref{eq:Hankel1a}) and (\ref{eq:Hankel2a})  are necessary for a Hankel tournament with score vector $R$. Now assume that (\ref{eq:Hankel1a}) and (\ref{eq:Hankel2a})  hold. By Lemma \ref{lem:nearly}  the nondecreasing arrangement of $R$
also satisfies (\ref{eq:Hankel2a}). Thus we need only check that the nondecreasing rearrangement of $R$ also satisfies (\ref{eq:Hankel1a}). Then the corollary follows from Theorem \ref{th:Hankel}. 

Let $x_1,x_2,\ldots,x_n$ be the nondecreasing rearrangement of $R$. Then for each $k$,  $x_k$ is the $k$th smallest of $r_1,r_2,\ldots,r_n$ and $x_{n+1-k}$ is the $k$th largest. Suppose that $x_k=a_j$. It follows from (\ref{eq:Hankel1a}) that $a_{n+1-j}$ is the $k$ largest of $a_1,a_2,\ldots,a_n$ and hence $x_{n+1-k}=n+1-a_j$. Therefore
$x_k+x_{n+1-k}=a_j+a_{n+1-j}=n-1$.
\end{proof}

 We now provide an algorithm to construct  a Hankel tournament with score vector $R$.
This algorithm is a variant of Ryser's algorithm \cite{Ry} to determine a tournament with a prescribed score vector (see also pages 220--222 of \cite{Br}). The last column of such a tournament determines the last row by  the combinatorial skew symmetry property of a tournament and thus, by the symmetry property of a Hankel tournament, determines the first row and first column.

\bigskip
\centerline{\bf Hankel Algorithm for a $T=[t_{ij}]\in {\mathcal T}_H(R)$ with $R$ nondecreasing}
\medskip
Let $R=(r_1,r_2,\ldots,r_n)$  be a nearly nondecreasing vector of nonnegative integers satisfying the Hankel property $r_i+r_{n+1-i}=n-1$ for $i=1,2,\ldots,n$  and the inequalities $\sum_{i=1}^kr_i\ge {k\choose 2}$ with equality for $k=n$ of (\ref{eq:Hankel2}).
We let $S_n=(s_1,s_2,\ldots,s_n)$  where $s_i=n-1-r_i$ for $i=1,2,\ldots,n$. 

\smallskip\noindent
Remarks: 
$S_n$ is the column sum vector of a Hankel tournament with score vector $R$. If $n$ is odd, then we know that $r_{(n+1)/2}=
(n-1)/2$, $r_n\ge (n-1)/2$, and $s_n\le (n-1)/2$.
 If $n$ is even, $r_n\ge n/2$  and  $s_n\le (n-2)/2$.

\begin{enumerate}
\item[(1)] If $n$ is even:
\begin{enumerate}
\item  Let $v=(v_1,v_2,\ldots,v_n)$ be the vector
$(\underbrace{0,\ldots,0}_{n-r_1-1},\underbrace{1,\ldots,1}_{r_1},0)$, and let $R_{n-2}'=(r_1',r_2',\ldots,r_{n-2}')$ be the 2-nearly nondecreasing vector with the  Hankel property defined by
\[r_i'=r_{i+1}-v_{i+1}-(1-v_{n-i}) \mbox{ for $1\le i\le n-2$}.\]
\item Let $Q_{n-2}$ be the permutation matrix of the type used in the proof of Lemma \ref{lem:monotone} such that $R_{n-2}'Q_{n-2}^t$ is a  nondecreasing sequence $R_{n-2}$ with the Hankel property.
\item Let $T_{n-2}$ be the $(n-2)\times (n-2)$ Hankel tournament obtained by this algorithm applied to the score vector $R_{n-2}$.
\item Let $T_n$ be the Hankel tournament with score vector $R$ defined by
\[T_n=\left[\begin{array}{c|c|c}
0&v_{n-1}\cdots v_2&v_1\\ \hline
1-v_{n-1}&&v_2\\
\vdots&Q_{n-2}^tT_{n-2}Q_{n-2}&\vdots \\
1-v_2&&v_{n-1}\\ \hline
1-v_1&1-v_2\cdots 1-v_{n-1}&0\end{array}\right].\]
\end{enumerate}

\item[(2)] If $n$ is odd:
\begin{enumerate}
\item Let $R_{n-1}'=(r_1',r_2',\ldots,r_{n-1}')$ 
be the nearly nondecreasing vector defined by
\[r_i'=\left\{\begin{array}{cl}
r_i,&\mbox{if $1\le i\le \frac{n-1}{2}$,}\\
r_{i+1}-1,&\mbox{if $\frac{n+1}{2}\le i\le n-1$.}\end{array}\right.\]
\item Let $Q_{n-1}$ be the permutation matrix of the type used in Lemma \ref{lem:monotone}  such that $R_{n-1}'Q_{n-1}^t$ is a 
nondecreasing vector $R_{n-1}$.
\item Let $T_{n-1}$ be the $(n-1)\times (n-1)$ Hankel tournament obtained by this algorithm 
applied to the score vector $R_{n-1}$, and let $T_{n-1}'$ be the Hankel tournament with score vector $R_{n-1}'$ obtained by applying the permutation matrix $Q_{n-1}$:
\[T_{n-1}'=Q_{n-1}^tT_{n-1} Q_{n-1}=\left[\begin{array}{c|c}
A&B\\ \hline J_{(n-1)/2}-B^{t}&A^{h}\end{array}\right].\]
\item Let $T_n$ be the Hankel tournament with score vector $R$ defined by
\[T_n=\left[\begin{array}{ccc|c|ccc}
&&&0&&&\\
&A&&\vdots&&B&\\
&&&0&&&\\ \hline
1&\cdots&1&0&0&\cdots&0\\ \hline
&&&1&&&\\
&J_{(n-1)/2}-B^{t} &&\vdots&&A^{h}&\\
&&&1&&&\end{array}\right].\]
\end{enumerate}
\item[(3)] Output $T_n$.
\end{enumerate}

\begin{example}{\rm Let $n=7$ and let $R=(1,2,2,3,4,4,5)$. Applying the Hankel algorithm, we obtain the following Hankel tournament in ${\mathcal T}_H(R)$:
\[\left[\begin{array}{c|c|c||c||c|c|c}
0&1&0&0&0&0&0\\  \hline
0&0&1&0&0&1&0\\ \hline
1&0&0&0&1&0&0\\ \hline\hline
1&1&1&0&0&0&0\\ \hline\hline
1&1&0&1&0&1&0\\ \hline
1&0&1&1&0&0&1\\ \hline
1&1&1&1&1&0&0\end{array}\right].\]
In carrying out the algorithm, the resulting score vectors are illustrated below where $\rightarrow_{\pi}$ means a permutation is used  and $\rightarrow_a$ means the resulting score vector after  a step of the algorithm:
\[(1,2,2,3,4,4,5)\rightarrow_a (1,2,2,3,3,4)\rightarrow_a (2,1,2,1)\rightarrow_{\pi}
(1,1,2,2) \rightarrow_a(1,0)\rightarrow_{\pi} (0,1).\]
}
\end{example}
\begin{theorem}\label{th:alghankel} The Hankel algorithm constructs a Hankel tournament in ${\mathcal T}_H(R)$ when $R$ satisfies the given conditions.
\end{theorem}

\begin{proof} We first assume that  $n$  is even. The entries of $R'$ satisfy:
\[r_i'=\left\{\begin{array}{cl}
r_{i+1},&\mbox{ if $1\le i\le r_1$},\\
r_{i+1}-1,&\mbox{ if $r_1+1\le i\le n-r_1-2$},\\
r_{i+1}-2,&\mbox{ if $n-r_1-1\le i\le n-2$.}\end{array}\right.\]
Since $R$ satisfies the Hankel property, so does $R'$.
Thus to  verify that the algorithm gives a Hankel tournament with score vector $R$, by Corollary \ref{cor:Hankel} we need only verify that $R_{n-2}'$ is 2-nearly nondecreasing and satisfies the corresponding inequalities of (\ref{eq:Hankel2}). 
Since $r_i'\le r_i$ and $r_j'\ge r_j-2$, we have $r_i'\le r_i\le r_j\le r_j'+2$ for $i<j$, and hence
$R_{n-2}'$
is 2-nearly nondecreasing.

We now verify that $R_{n-2}'$ satisfies the corresponding inequalities of (\ref{eq:Hankel2}).
Again we consider three cases,

\smallskip\noindent
Case $1\le k\le r_1$: We have
\[\sum_{i=1}^kr_i'=\sum_{i=1}^k r_{i+1}\ge \sum_{i=1}^kr_i\ge {k\choose 2}.\]

\smallskip
\noindent
Case $r_1+1\le k\le n-r_1-2$: We calculate that
\begin{eqnarray*}
\sum_{i=1}^k r_i'&=&\sum_{i=1}^{r_1}r_{i+1}+
\sum_{i=r_1+1}^k (r_{i+1}-1)\\
&=&\left(\sum_{i=1}^kr_{i+1}\right)-(k-r_1)
=\left(\sum_{i=1}^{k+1}r_i\right)-k\\
&\ge& {{k+1}\choose 2}-k={k\choose 2}.
\end{eqnarray*}

\smallskip
\noindent
Case $n-r_1-1\le k\le n-2$: We first observe that for $n-r_1\le l\le n-1$, we have the stronger inequality
\[\sum_{i=1}^lr_i\ge {l\choose 2} +(l-n+r_1+1)\]
by considering  the minimum number of 1s in the $l\times (l+1)$ 
submatrix determine by rows $1,2,\ldots,l$ and columns $1,2,\ldots,l,n$.
Using this inequality we calculate that
\begin{eqnarray*}
\sum_{i=1}^kr_i' &=& \sum_{i=1}^{r_1}r_{i+1}+\sum_{i=r_1+1}^{n-r_1-2} (r_{i+1}-1)+\sum_{i=n-r_1-1}^k (r_{i+1}-2)\\
&=& \left(\sum_{i=1}^k r_{i+1}\right)-(n-2r_1-2)-2(k-n+r_1+2)\\
&=&\left(\sum_{i=1}^{k+1} r_i\right)- (2k+2+r_1-n)\\
&\ge& {{k+1}\choose 2}+(k-n+r_1+2)-(2k+2+r_1-n)\\
&=& {{k+1}\choose 2}-k={k\choose 2}.
\end{eqnarray*}

Now we  assume that  $n$  is odd.

The entries of the vector $R_{n-1}'$ are given by
\[r_i'=\left\{\begin{array}{cl}
r_i,& \mbox{ if $1\le i\le \frac{n-1}{2}$,}\\
r_{i+1}-1,&\mbox{ if $\frac{n+1}{2}\le i\le n-1$.}\end{array}\right. \]
It follows by inspection that the vector $R_{n-1}'$ is nearly nondecreasing, and we have $r_i'+r_{n-i}'=(n-1)-1$ for each $i$. Now we verify that  Landau's inequalities 
$\sum_{i=1}^kr_i'\ge {k\choose 2}$ hold for $1\le k\le n-1$ with equality if $k=n-1$. 
If $1\le k\le (n-1)/2$, then this is clear. Suppose that  $(n-1)/2\le k\le n-1$. Then 
\begin{eqnarray*}
\sum_{i=1}^k r_i'&=&\left(\sum_{i=1}^{(n-1)/2}r_i\right)+\left(\sum_{i=(n+3)/2}^{k+1} r_i\right)-
\left(k+1-\frac{n+3}{2}+1\right)\\
&=& \left(\sum_{i=1}^{k+1}r_i\right)-k+\frac{n-1}{2}-r_{(n+1)/2}\\
&=&\left( \sum_{i=1}^{k+1}r_i\right)-k
\ge  {{k+1}\choose 2}-k={k\choose 2},
\end{eqnarray*}
with equality if $k=n-1$.
Since $R'$ is nearly nondecreasing, after rearranging its terms to form a nondecreasing vector, we can apply the algorithm for the even case of $n-1$ and as indicated in the algorithm   construct a Hankel tournament with score vector $R$.
\end{proof}

Let $T^*$ be the Hankel tournament in ${\mathcal T}_H(R)$ constructed by the Hankel algorithm.
We next identify certain switches and pairs of switches that allow one to move from any Hankel tournament $T\in{\mathcal T}_H(R)$  to $T^*$ where each switch  and pairs of switches produces another Hankel  tournament in ${\mathcal T}_H(R)$. Since these switches are reversible,  this allows one to move from any   $T_1\in {\mathcal T}_H(R)$ to any other
$T_2\in {\mathcal T}_H(R)$ where each switch produces a Hankel tournament in  ${\mathcal T}_H(R)$. 
We collect all these switches, including the switches used for loopy tournament and a switch to be used for combinatorially skew-Hankel tournaments, in Table 1.

\begin{center}\begin{tabular}{|c|c|c|}\hline
edge-loop switch & $\hbox{$\rightarrow$}\kern -1.5pt\hbox{$\circ$}_{ij}$ &
$\begin{array}{c||c|c} &i&j \\ \hline \hline i&0&1 \\ \hline j&0&1  \end{array} \rightarrow
\begin{array}{c||c|c} &i&j \\ \hline \hline i&1&0 \\ \hline j&1&0  \end{array}$ \\ \hline
3-cycle switch & $\triangle_{i,j,k}$ &
$\begin{array}{c||c|c|c} &i&j&k \\ \hline \hline i&x&1&0 \\ \hline j&0&y&1 \\ \hline k&1&0&z\end{array} \rightarrow
\begin{array}{c||c|c|c} &i&j&k \\ \hline \hline i&x&0&1 \\ \hline j&1&y&0 \\ \hline k&0&1&z \end{array}$ \\ \hline
3-cycle switch & $\triangle_{i,j,k}$ & no loops:  $x=y=z=0$ above \\ \hline
\begin{tabular}{c}
3-cycle Hankel\\  switch \end{tabular}& $\triangle_{i,\frac{n+1}{2},n+1-i}$ &
$j=(n+1)/2$ and $k=n+1-i$ above \\ \hline
4-cycle switch & $\square_{i,j,k,l}$ &
$\begin{array}{c||c|c|c|c} &i&j&k&l \\ \hline \hline i&0&1&a&0 \\ \hline j&0&0&1&b \\ \hline k&1-a&0&0&1 \\ \hline l&1&1-b&0&0 \end{array} \rightarrow \begin{array}{c||c|c|c|c} &i&j&k&l \\ \hline \hline i&0&0&a&1 \\ \hline j&1&0&0&b \\ \hline k&1-a&1&0&0 \\ \hline l&0&1-b&1&0 \end{array}$ \\ \hline
\begin{tabular}{c}
4-cycle Hankel \\switch\end{tabular} & $\square_{i,j,n+1-j,n+1-i}$ & $k=n+1-j$, $l=n+1-i$, and $a=b$ above \\ \hline
\begin{tabular}{c} 
4-cycle skew-Hankel\\ switch\end{tabular} & $\square_{i,j,n+1-i,n+1-j}$ & 
$\begin{array}{c||c|c|c|c} &i&j&k&l\\ \hline \hline i&0&1&0&0 \\ \hline j&0&0&1&0 \\ \hline k&0&0&0&1 \\ \hline l&1&0&0&0 \end{array} \rightarrow \begin{array}{c||c|c|c|c} &i&j&k&l\\ \hline \hline i&0&0&0&1 \\ \hline j&1&0&0&0\\ \hline k&0&1&0&0 \\ \hline l&0&0&1&0 \end{array}$  $\begin{array}{c}k=n+1-i\\ l=n+1-j\end{array}$
 \\ \hline
\end{tabular}\end{center}

\medskip
\centerline{Table 1: Summary of switches.}

Let $i,j,k,l$ be four distinct indices and consider the $4\times4$  matrix $T[i,j,k,l]$ of $T$ 
whose row and column indices are $\{i,j,k,l\}$. 
If we have a directed cycle $i\rightarrow j\rightarrow k\rightarrow l\rightarrow i$ in a tournament $T$, then $T[i,j,k,l]$ has the form on the left in (\ref{eq:spring}), and  we define $\square_{i,j,k,l}$ to be the switch that replaces $T[i,j,k,l]$  with the matrix on the right:
\begin{equation}\label{eq:spring}
\begin{array}{c||c|c|c|c} &i&j&k&l \\ \hline \hline i&&1&a&0 \\ \hline j&0&&1&b \\ \hline k&1-a&0&&1 \\ \hline l&1&1-b&0&\end{array} \rightarrow \begin{array}{c||c|c|c|c} &i&j&k&l \\ \hline \hline i&&0&a&1 \\ \hline j&1&&0&b \\ \hline k&1-a&1&&0 \\ \hline l&0&1-b&1&  \end{array}\ .\end{equation}
The switch $\square_{i,j,k,l}$ reverses a 4-cycle of $T$: $$(i\rightarrow j\rightarrow k\rightarrow l\rightarrow i) \rightarrow (i\leftarrow j\leftarrow k\leftarrow l\leftarrow i)$$ and we call it a \emph{$4$-cycle switch}. It is easy to see that this reversal of a 4-cycle can also be accomplished
by the reversal of two 3-cycles: if, for instance, we have $a=b=0$, we reverse two 3-cycles using the \emph{$3$-cycle switches} $\triangle_{i,j,k}$ and then $\triangle_{i,k,l}$.
We define $\triangle^c_{i,j,k}$ to be the switch $\triangle_{n+1-i,n+1-j,n+1-k}$ and we refer to it as the \emph{complementary switch} of $\triangle_{i,j,k}$ (the indices are  complementary to $i,j,k$). In the same way, we define $\square^c_{i,j,k,l}$ to be the switch $\square_{n+1-i,n+1-j,n+1-k,n+1-l}$ and call it the \emph{complementary switch} of $\square_{i,j,k,l}$.

For Hankel tournaments, we have that $$T[i,j,k,l]=T[n+1-l,n+1-k,n+1-j,n+1-i]^h.$$ We will consider two possibilities for the indices $i,j,k,l$:
\smallskip

(i) $j\neq n+1-i,n+1-k$ and  $l\neq n+1-i,n+1-k$ with at most one of the equalities $k=n+1-i$ and  $l=n+1-j$; \smallskip

(ii)  $k=n+1-j$ and $l=n+1-i$.

\smallskip\noindent
In case (i), there is no overlap on the entries of $T[i,j,k,l]$ and $T[n+1-l,n+1-k,n+1-j,n+1-i]$ in the positions we are changing when applying $\square_{i,j,k,l}$ and $\square^c_{l,k,j,i}$. We say that $\square_{i,j,k,l}$ is a \emph{pure $4$-cycle switch}. Applying the  pure 4-cycle switch $\square_{i,j,k,l}$ and the complementary switch in the reverse order $\square^c_{l,k,j,i}=\square_{n+1-l,n+1-k,n+1-j,n+1-i}$ to a $T\in\mathcal{T}_H(R)$ results in another Hankel tournament in $\mathcal{T}_H(R)$.

In case (ii), the matrix $T[i,j,n+1-j,n+1-i]$ has the form 
$$\begin{array}{c||c|c|c|c}
&i&j&n+1-j&n+1-i \\ \hline \hline
i&0&1&a&0 \\ \hline
j&0&0&1&a \\ \hline
n+1-j&1-a&0&0&1 \\ \hline
n+1-i&1&1-a&0&0
\end{array}.$$
If we apply $\square_{i,j,n+1-j,n+1-i}$ to a $T\in\mathcal{T}_H(R)$ we obtain another Hankel tournament in $\mathcal{T}_H(R)$. In this case $\square_{i,j,n+1-j,n+1-i}=\square^c_{n+1-i,n+1-j,j,i}$ and we just need to apply one 4-cycle switch.

If $n$ is odd and $j=\frac{n+1}{2}$, then $T[i,j,n+1-j,n+1-i]$ collapses to the $3\times3$ matrix $T[i,\frac{n+1}{2},n+1-i]$, and if there is a  3-cycle we can apply the 3-cycle switch $\triangle_{i,\frac{n+1}{2},n+1-i}$, which reverses the orientation of the 3-cycle. We call such a switch a \emph{$3$-cycle Hankel switch}, when we have these  special indices. Applying a 3-cycle Hankel switch $\triangle_{i,\frac{n+1}{2},n+1-i}$ to a $T\in\mathcal{T}_H(R)$ results in another Hankel tournament in $\mathcal{T}_H(R)$.

When $R$ is nondecreasing, we now show how to use these switches to move  from any tournament $T\in\mathcal{T}_H(R)$ to $T^*\in\mathcal{T}_H(R)$ never leaving the class ${\mathcal T}_H(R)$. We consider two cases according to whether $n$ is even or odd.

\smallskip\noindent
{\bf Case $n$ even}:
 If the  first column of $T$ equals the first column of $T^*$, then their last columns, first rows, and last rows are equal, respectively.  Thus the {\it borders} of $T$ and $T^*$ agree; deleting their borders, we proceed by induction on $n$. Now suppose that the  borders of $T$ and $T^*$ are not equal. The first column of $T^*$ consists of $n-s_1=r_1+1$ 0s followed $s_1$ 1s. Since $t_{11}=0$,  there is a first $i\ne 1$ such that  $t_{i1}=1$ and $t_{j1}=0$ for some $j>i$.

First suppose that we may take $j=n$, that is, $t_{n1}=0$. Since $r_1$ is the smallest score, and $r_1+s_1=n-1$, we have that $s_1\ge (n/2)$.  Since $t_{n1}=0$, it follows that there exists an $1<i\le (n/2)$ such that $t_{i1}=t_{n+1-i,1}=1$. Since $T$ is a Hankel tournament,  we have
\[T[1,i,n+1-i,n]=\left[\begin{array}{c|c|c|c}
0&0&0&1\\ \hline
1&0&a&0\\ \hline
1&1-a&0&0\\ \hline
0&1&1&0\end{array}\right],\]
that is,
\[(a=0)\ \left[\begin{array}{c|c|c|c}
0&0&0&1\\ \hline
1&0&0&0\\ \hline
1&1&0&0\\ \hline
0&1&1&0\end{array}\right] \mbox{ or } (a=1)\ 
\left[\begin{array}{c|c|c|c}
0&0&0&1\\ \hline
1&0&1&0\\ \hline
1&0&0&0\\ \hline
0&1&1&0\end{array}\right],\]
In  the case that $a=0$,  we apply the $4$-cycle Hankel switch $\Box_{1,n,n+1-i,i}$ and the result  is a  tournament in ${\mathcal T}_H(R)$ in which $t_{i1}$ has been replaced with 0 and $t_{n1}$ has been replaced with 1.  In the case that $a=1$, we apply the 4-cycle Hankel switch $\Box_{1,n,i,n+1-i}$ and the result is a   tournament in ${\mathcal T}_H(R)$ in which $t_{n+1-i,1}$ has been replaced with 0 and $t_{n1}$ has been replaced with 1.

Now assume that $t_{n1}=1$ so that there exists  a $j$ with $i<j<n$ such that $t_{j1}=0$. First suppose that we may take $j=n+1-i$. Since $r_i<r_{n+1-i}$ there exists a $k$ with $k\ne 1,n+1-i$, such that $t_{ik}=0$ and $t_{n+1-i,k}=1$. Since $t_{i1}=1$ and $T$ is a Hankel tournament,  $t_{n,n+1-i}=1$  and hence $t_{n+1-i,n}=0$. Therefore $k\ne n$. In addition, because $r_i<r_{n+1-i}$, we may choose $k\ne i$. 
We now have
\[T[1,i,n+1-i,k]=\left[\begin{array}{c|c|c|c}
0&0&1&a\\ \hline
1&0&b&0\\ \hline
0&1-b&0&1\\ \hline
1-a&1&0&0\end{array}\right].\]
The switch $\Box_{1,n+1-i,k,i}$ is a pure 4-cycle switch, and $\Box_{1,n+1-i,k,i}$ followed by  $\Box_{i,k,n+1-i,1}^{c}$
produces a tournament in ${\mathcal T}_H(R)$ in which $t_{i1}$ has been replaced with 0 and $t_{n+1-i,1}$ has been replaced with 1.

The remaining possibility in the  $n$ even case is that   $t_{i1}=1$ implies that $t_{n+1-i,1}=1$, so we cannot take $j=n+1-i$. If $t_{ji}=1$, then
\[T[1,i,j]=\left[\begin{array}{c|c|c}
0&0&1\\ \hline
1&0&0\\ \hline
0&1&0\end{array}\right].\]
The  switch $\triangle_{1,j,i}$ 
is a pure 3-cycle switch, and $\triangle_{1,j,i}$ 
followed by $\triangle_{i,j,1}^c$ produces a tournament in ${\mathcal T}_H(R)$ in which $t_{i1}$ has been replaced with 0 and $t_{j1}$ has been replaced with 1.
Now suppose that $t_{ji}=0$ so that $t_{ij}=1$. Since $r_i\le r_j$, there exists at least two values of $k$  such that $k\ne 1,i,j$, and  $t_{ik}=0$ and $t_{jk}=1$. If $t_{j,n+1-i}=0$, then there exists such a $k\ne n+1-i,n+1-j$. If $t_{j,n+1-i}=1$, then  since $T$ is a Hankel tournament, $t_{i,n+1-j}=1$ and hence $k\ne n+1-j$, and since there are two possible values of  $k$, we may choose $k\ne n+1-i$.
Thus we have
\[T[1,i,j,k]=\left[\begin{array}{c|c|c|c}
0&0&1&a\\ \hline
1&0&1&0\\ \hline
0&0&0&1\\ \hline
1-a&1&0&0\end{array}\right].\]
The switch $\Box_{1,j,k,i}$ is a pure 4-cycle switch, and $\Box_{1,j,k,i}$ followed by  $\Box_{i,k,j,1}^{c}$
produces a tournament in ${\mathcal T}_H(R)$ in which $t_{i1}$ has been replaced with 0 and $t_{j1}$ has been replaced with 1.

Hence by induction, if $n$ is even, we can move from $T$ to $T^*$ by a sequence of $4$-cycle Hankel switches, and  pairs of switches consisting of a  pure $4$-cycle switch $\square_{i,j,k,l}$ and its complementary switch  in the reverse order
$\square_{l,k,j,i}^c$.

\smallskip\noindent
{\bf Case $n$ odd}:
 Then row $(n+1)/2$ of the  matrix $T^*$ constructed by the Hankel algorithm 
contains $(n-1)/2$ $1$s and they occur in its  first $(n-1)/2$ columns, thus determining both row and column $(n+1)/2$. We show that by $3$-cycle Hankel switches and pairs of switches consisting of  a pure $4$-cycle switch and the complementary switch in the reverse order, we can bring  any tournament $T\in {\mathcal T}(R)$ into a tournament $T'\in {\mathcal T}(R)$ which agrees with $T^*$ on row $(n+1)/2$, that is, has all its 1s in columns $1,2,\ldots,(n-1)/2$. We may then delete row and column $(n+1)/2$ of $T'$ and $T^*$ leaving  two $(n-1)\times (n-1)$ Hankel tournaments  $T''$ and $T^{**}$ with the same score vector $R'$. By what we have proved, $T''$ can be brought to $T^{**}$ 
by a sequence of $4$-cycle Hankel switches and  pairs of switches consisting of  a pure $4$-cycle switch and its complementary switch in the reverse order. It
 then follows that $T$ can be brought to $T^*$ by a sequence of 
 $3$-cycle Hankel switches, $4$-cycle Hankel switches, and pairs of switches consisting of  a pure $4$-cycle switch and its complementary switch in the reverse order. 
 
So suppose that $t_{(n+1)/2,i}=0$ for some $i\le (n-1)/2$. 
We consider the digraph $D$ with vertex set $\{1,2,\ldots,n\}$ with an edge from $i$ to $j$ if and only if $t_{ij}=1$ and $t_{ij}^*=0$. Since $T$ and $T^*$ have the same score vector, the outdegree of a vertex equals its indegree and hence 
the edges of $D$ can be partitioned into cycles. One of these cycles $C$  uses the vertex $(n+1)/2$:
\[\frac{n+1}{2}\rightarrow i_1\rightarrow i_2\rightarrow\cdots\rightarrow i_k\rightarrow \frac{n+1}{2}.\]
Since we are dealing with Hankel tournaments, $D$ also contains the edge $(n+1-i_1)\rightarrow \frac{n+1}{2}$. If $T$ also contains the edge $i_1\rightarrow (n+1-i_1)$, then we can apply the $3$-cycle Hankel switch  $\bigtriangleup_{i_1,\frac{n+1}{2},n+1-i_1}$ and move $T$ closer to $T^*$. If $T$ contains the edge $ (n+1-i_1)\rightarrow i_1$ (and so does not contain the edge $i_1\rightarrow (n+1-i_1)$), then we consider the edge $i_1\rightarrow i_2$. Then $n+1-i_2\rightarrow n+1-i_1$ is also an edge of $D$. If $i_2\rightarrow (n+1-i_2)$, then the $4$-cycle Hankel switch
$\Box_{n+1-i_1,i_1,i_2,n+1-i_2}$ reverses $(n+1-i_1)\rightarrow i_1$, and then we can apply $\bigtriangleup_{i_1,\frac{n+1}{2},n+1-i_1}$  as above. If, on the other hand, $(n+1-i_2)\rightarrow i_2$ is an edge of $T$, we consider the pair of vertices $(n+1-i_1)$ and $i_2$ (see Figure 1), and proceed as follows.

 If $i_2\rightarrow (n+1-i_1)$ is an edge of $T$, then we apply the 4-cycle Hankel switch $\Box_{i_2,n+1-i_1,i_1,n+1-i_2}$ to reverse $(n+1-i_1)\rightarrow i_1$ and then  $\bigtriangleup_{i,\frac{n+1}{2},n+1-i_1}$ as above. 

If $(n+1-i_1)\rightarrow i_2$ is an edge of $T$, then we consider the edge $i_2\rightarrow i_3$ where now $i_3\ne (n+1-i_1),(n+1-i_2)$.

\begin{center}\begin{tikzpicture}
  [scale=1,auto=left,every node/.style={circle,draw,scale=0.8}]
  \node[label=above:$\frac{n+1}{2}$] (0) at (0,1) {};
  \node[label=right:$i_1$] (1) at (1,0) {};
  \node[label=left:$n+1-i_1$] (2) at (-1,0) {};
  \node[label=right:$i_2$] (3) at (1,-1) {};
  \node[label=left:$n+1-i_2$] (4) at (-1,-1) {};
  \node[label=right:$i_3$] (5) at (1,-2) {};
  \draw[->-] (0) to (1);
  \draw[->-] (2) to (0);
  \draw[dashed,->-] (2) to (1);
  \draw[->-] (1) to (3);
  \draw[->-] (4) to (2);
  \draw[dashed,->-] (4) to (3);
  \draw[dashed,-->-] (3) to (2);
  \draw[dashed,-->-] (1) to (4);
  \draw[->-] (3) to (5);
\end{tikzpicture}\end{center}

\medskip
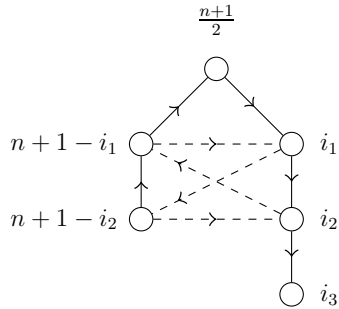
\centerline{Figure 1: Part of the digraph $D$.}

\noindent
We continue like this until we are able to find an $i_j\ne n+1-i_1,\ldots,n+1-i_{j-1}$ such that $i_j
\rightarrow (n+1-i_h)$ for some $h\le j$. There are two possibilities to consider. If we also have the edge $(n+1-i_j)\rightarrow i_j$ (only if $h<j$) , then we apply the $4$-cycle Hankel switch $\Box_{i_h,n+1-i_j, i_j,n+1-i_h}$ to reverse $(n+1-i_h)\rightarrow i_h$, followed by the $4$-cycle Hankel switches $\Box_{i_{l-1}i_l,n+1-i_l,n+1-i_{l-1}}$ with $l=h,h-1,\ldots,2$ to finally reverse $(n+1-i_1)\rightarrow i_1$, and then $\bigtriangleup_{i,\frac{n+1}{2},n+1-i_1}$.
If for all $j$ with $2\le j\le k$, $i_j\rightarrow (n+1-i_h)$ is not an edge of $T$ for a $h$ with $1\le h\le j$, then since $i_k\rightarrow
(n+1)/2$ is an edge of cycle $C$, the cycle $C$ does not contain both a vertex $p$ and a vertex $n+1-p$. Thus using a sequence of pairs of pure 3-cycle switches and their complementary 3-cycle switches in the reverse order, we can reverse $C$ and its complementary cycle of opposite orientation. By these operations we have obtained from $T$ a Hankel tournament in ${\mathcal T}_H(R)$
which agrees with $T^*$ in row and column $(n+1)/2$.

 Putting the preceding arguments together and having in mind that a pure 4-cycle switch can be accomplished by two 3-cycle switches,  we have proved the following theorem.

\begin{theorem}\label{th:switch}
Let $T_1$ and $T_2$ be two tournaments in ${\mathcal T}_H(R)$. Then there exists a sequence of moves consisting of $4$-cycle Hankel switches, pairs consisting of a pure $3$-cycle switch
and its complementary $3$-cycle switch in the reverse order, and $3$-cycle Hankel switches, which  brings $T_1$ to $T_2$ with all intermediary tournaments in ${\mathcal T}_H(R)$.
\end{theorem}
 
\smallskip\noindent
\begin{example}{\rm Let $n=7$ and $R=(1,2,2,3,4,4,5)$, and consider the two Hankel tournaments $T_1$ and $T_2$ in ${\mathcal T}_H(R)$ given by
\[
T_1=\left[\begin{array}{c|c|c|c|c|c|c}
0&0&0&0&0&0&1\\ \hline
1&0&0&0&0&1&0\\ \hline
1&1&0&0&0&0&0\\ \hline
1&1&1&0&0&0&0\\ \hline
1&1&1&1&0&0&0\\ \hline
1&0&1&1&1&0&0\\ \hline
0&1&1&1&1&1&0\end{array}\right]
\mbox{ and }
T_2=\left[\begin{array}{c|c|c|c|c|c|c}
0&1&0&0&0&0&0\\ \hline
0&0&1&1&0&0&0\\ \hline
1&0&0&0&1&0&0\\ \hline
1&0&1&0&0&1&0\\ \hline
1&1&0&1&0&1&0\\ \hline
1&1&1&0&0&0&1\\ \hline
1&1&1&1&1&0&0\end{array}\right].
\]
We can move from $T_1$ to $T_2$ by the 4-cycle Hankel switch
$\square_{7,5,3,1}$, followed by  the pair of pure 3-cycle switches $\bigtriangleup_{3,2,1}$ and $\bigtriangleup_{1,2,3}^c=\bigtriangleup_{7,6,5}$,  followed by the 3-cycle Hankel switch
$\bigtriangleup_{6,4,2}$.

}
\end{example}

For completeness we mention the following. It is natural as well to  consider $n\times n$  {\it Hankel loopy tournaments} $T=[t_{ij}]$, that is, Hankel tournaments
with possible $1$s on the main diagonal. By the Hankel property, we now have that  for all $i$,
$t_{ii}=t_{n+1-i,n+1-i}$ where the common value is either 0 or 1.
The score vector $R=(r_1,r_2,\ldots,r_n)$ of a  Hankel loopy tournament satisfies
$r_i+r_{n+1-i}= n+1\mbox{ or } n-1$  depending on whether or not $t_{ii}=t_{n+1-i,n+1-i}=1$.
Thus the score vector of a  Hankel loopy tournament determines which elements on the main diagonal equal 1 and which equal 0. Thus if $R'$ is obtained from $R$ by subtracting 1 from those $i$ and $n+1-i$ for which $r_i+r_{n+1-i}=n+1$, then there is a  Hankel loopy tournament with score vector $R$ if and only if there is a Hankel tournament with score vector $R'$. 

\begin{example}{\rm The  Hankel loopy tournament
\[T=\left[\begin{array}{c|c|c|c|c}
0&0&0&1&1\\ \hline
1&1&0&0&1\\ \hline
1&1&1&0&0\\ \hline
0&1&1&1&0\\ \hline
0&0&1&1&0\end{array}\right]\]
with score vector $R=(2,3,3,3,2)$ corresponds to the Hankel tournament
\[T'=\left[\begin{array}{c|c|c|c|c}
0&0&0&1&1\\ \hline
1&0&0&0&1\\ \hline
1&1&0&0&0\\ \hline
0&1&1&0&0\\ \hline
0&0&1&1&0\end{array}\right]\]
with score vector $R'=(2,2,2,2,2)$.
}
\end{example}

\section{Combinatorially skew-Hankel Tournaments}

Let $T=[t_{ij}]$ be an $n\times n$ combinatorially skew-Hankel tournament, and let the score vector of $T$ be $R=(r_1,r_2,\ldots,r_n)$.  Since $T$ is invariant under a rotation by 180 degrees, we have that for each $i$, row $n+1-i$ is obtained by reversing row $i$, and thus  the score vector of $T$ is {\it palindromic}, that is,
\[R=(r_1,r_2,r_3,\ldots,r_3,r_2,r_1).\]
Since $T$ is combinatorially skew with respect to the main diagonal, row $i$ of $T$ not only determines row $n+1-i$ but it  also determines columns $i$ and $n+1-i$.
We illustrate this important property in the next example.

\begin{example}{\rm
Let $T$ be an $8\times 8$ combinatorially skew-Hankel tournament with row 3 given.
Then the entries in column 3, and row and column $8+1-3=6$ are as shown in
\[\left[\begin{array}{c|c||c||c|c||c||c|c}
\phantom{11}0\phantom{11}&\phantom{1-a}&1-a&\phantom{1-a}&\phantom{1-a}&1-f&\phantom{1-1}&\phantom{11}0\phantom{11}\\ \hline
&0&1-b&&&1-e&0&\\ \hline\hline
a&b&0&c&d&0&e&f\\ \hline\hline
&&1-c&0&0&1-d&&\\ \hline
&&1-d&0&0&1-c&&\\ \hline\hline
f&e&0&d&c&0&b&a\\ \hline\hline
&0&1-e&&&1-b&0&\\ \hline
0&&1-f&&&1-a&&0\end{array}\right],\]
where the 0s on the main and Hankel diagonals have been inserted.}
\end{example}

The score vector of a combinatorially skew-Hankel tournament is determined by $n/2$ integers if $n$ is even, and by $(n+1)/2$ integers if $n$ is odd. If we reorder the  rows $1,2,\ldots,\lfloor n/2\rfloor$ of $T$ and then reorder rows $\lceil 
n/2\rceil, \ldots, n-1,n$ in the same way, and similarly reorder columns $1,2,\ldots,\lfloor n/2\rfloor$ and columns $\lceil  n/2\rceil, \ldots,n-1,n$ in the corresponding way, the result is another combinatorially skew-Hankel tournament. Hence there is no loss of generality in assuming that the first half $(r_1,r_2,\ldots,r_{\lfloor n/2\rfloor})$ of $R$ is nondecreasing and thus that the second half $(r_{\lfloor n/2\rfloor},\ldots,r_2,r_1)$ is nonincreasing.  If $n$ is odd, $R$ has a middle term $r_{(n+1)/2}$ which is even 
since $t_{(n+1)/2,(n+1)/2}=0$,  and both row  $(n+1)/2$ and column $(n+1)/2$  of $T$ are palindromic. 

\begin{example}{\rm A $7\times 7$ combinatorially
 skew-Hankel tournament  with palindromic score vector $R=(1,3,4,2,4,3,1)$ is given by
\[\left[\begin{array}{c|c|c|c|c|c|c}
0&0&1&0&0&0&0\\ \hline
1&0&0&1&0&0&1\\ \hline
0&1&0&1&0&1&1\\ \hline
1&0&0&0&0&0&1\\ \hline
1&1&0&1&0&1&0\\ \hline
1&0&0&1&0&0&1\\ \hline
0&0&0&0&1&0&0\end{array}\right].\]

}
\end{example}

We first treat the case where $n$ is even.  In this case a combinatorially skew-Hankel tournament is of the form
\[
\left[\begin{array}{cc}
T_1&T_2\\
T_2^{th}& T_1^{ht}\end{array}\right]\]
where $T_1$ is an $\frac{n}{2}\times\frac{n}{2}$ tournament and $T_2$, when its columns are taken from last to first,
is an $\frac{n}{2}\times\frac{n}{2}$ tournament.

Let 
$R=(r_1,r_2,\ldots,r_{n/2},r_{n/2},\ldots,r_2,r_1)$
satisfy $r_1\le r_2\le\cdots\le r_{n/2}$. The nondecreasing rearrangement of $R$ is
$R^*=(r_1^*,r_2^*,\ldots,r_n^*)=(r_1,r_1,r_2,r_2,\ldots,r_{n/2},r_{n/2})$ where therefore
$r^*_{2k-1}=r^*_{2k}=r_k$ for $k=1,2,\ldots,n/2$. Since in an $n\times n$ combinatorially skew-Hankel tournament  $T=[t_{ij}]$, $t_{kk}=t_{k,n+1-k}=t_{n+1-k,n+1-k}=t_{n+1-k,k}=0$, the sum of the $l$ smallest elements of $R$, that is, the first $l$ elements of $R^*$, must be at least the number 
${l\choose 2}-\left\lfloor\frac{l}{2}\right\rfloor$
of games played amongst themselves, and thus
\begin{equation}\label{eq:self}
\sum_{i=1}^l r_i^*\ge {l\choose 2}-\left\lfloor\frac{l}{2}\right\rfloor,\ (l=1,2,\ldots,n),  \mbox{ with equality if $l=n$.}\end{equation}

\begin{lemma}\label{lem:skew1}
Let  $h$,  with $1\le h\le n/2$, be such that 
equality holds in $(\ref{eq:self})$  for $l=2h-1$. Then equality also holds in $(\ref{eq:self})$ for $l=2h$.
\end{lemma}

\begin{proof}
We have
$\sum_{i=1}^{2h-1} r_i^*= {2h-1\choose 2}-(h-1)$. If $h=1$, then $r_1^*=0$ and hence $r_2^*=0$,  and so equality holds for $l=2$.
Now let $h\ge 2$. Then
\[r_{2h}^*=\sum_{i=1}^{2h}r_i^*-\sum_{i=1}^{2h-1}r_i^*\ge 
{{2h}\choose 2}-h-\left({{2h-1}\choose 2}-(h-1)\right)=2h-2,\]
and
\[r_{2h-1}^*=\sum_{i=1}^{2h-1}r_i^*-\sum_{i=1}^{2h-2}r_i^*\le 
\left({{2h-1}\choose 2}-(h-1)\right)-\left({{2h-2}\choose 2}-(h-1)\right)=2h-2.\]
Thus
$2h-2\le r_{2h}^*=r_h=r_{2h-1}^*\le 2h-2$, and we conclude that
$r_{2h}^*=2h-2$. With this value for $r_{2h}^*$, we get
\[\sum_{i=1}^{2h}r_i^*=\left(\sum_{i=1}^{2h-1}r_i^*\right)+r_{2h}^*=
{{2h-1}\choose 2}-(h-1)+(2h-2)={{2h}\choose 2}-h,\]
and thus equality holds in (\ref{eq:self}).
\end{proof}

By taking $l=2k$ in (\ref{eq:self}), we see that  the inequalities (\ref{eq:self}) for $l$ even are equivalent to
\begin{equation}\label{eq:self2}
\sum_{i=1}^kr_i\ge k(k-1),\ (k=1,2,\ldots,\frac{n}{2}), \mbox{ with equality for $k=\frac{n}{2}$.}\end{equation}
In the next lemma we show that (\ref{eq:self}) is equivalent to (\ref{eq:self2}) even if $l$ is odd.
  
  \begin{lemma}\label{lem:skew2}
  The vector $R^*$ satisfies $(\ref{eq:self})$ if and only if $R$ satisfies $(\ref{eq:self2})$.
  \end{lemma}
  
  \begin{proof}
  The inequalities (\ref{eq:self}) for $l$ even  are equivalent to  the inequalities
  (\ref{eq:self2}) for $k=1,2,\ldots, n/2$. Thus  we need only show that if the inequalities (\ref{eq:self2}) hold, then the inequalities (\ref{eq:self}) hold for odd $l$.
  Let $l=2k+1$ and suppose to the contrary that
  $\sum_{i=1}^{2k+1}r_i^*<{{2k+1}\choose 2}-k$. Then   
  \[r_{2k+1}^*=\sum_{i=1}^{2k+1}r_i^*-\sum_{i=1}^{2k}r_i^*
  <\left({{2k+1}\choose 2}-k\right)-\left({{2k}\choose 2}-k\right)=2k.
\]
  Since $r_{2k+2}^*=r_{2k+1}^*$, we obtain
  \[\sum_{i=1}^{2k+2}r_i^*=\left(\sum_{i=1}^{2k+1}r_i^*\right)+r_{2k+2}^*<
 \left( {{2k+1}\choose 2}-k\right)+2k={{2k+2}\choose 2}-(k+1),\]
  and this contradicts (\ref{eq:self}) and (\ref{eq:self2}).
  \end{proof}
  
  We now verify that it is enough for us  to assume that
   $(r_1,r_2,\ldots,r_{(n+1)/2})$ is 3-nearly nondecreasing.

  \begin{lemma}\label{lem:skew3}
  Let  $(r_1,r_2,\ldots,r_{m})$ be a sequence of nonnegative integers that satisfies $(\ref{eq:self2})$ such that $R$ is $3$-nearly nondecreasing. Then its nondecreasing rearrangement also satisfies $(\ref{eq:self2})$.
  \end{lemma}
  
  \begin{proof}
  It is enough to show that if for some $k$, $r_k=a$ and $r_{k+1}=a-p$ where $1\le p\le 3$, then interchanging $r_k$ and $r_{k+1}$ results in a vector also satisfying (\ref{eq:self2}).  If this was false, then we have that
  $\sum_{i=1}^kr_i=k(k-1)+h$ where $0\le h\le p-1$. Then
  \[r_{k+1}=\sum_{i=1}^{k+1}r_i-\sum_{i=1}^kr_i\ge (k+1)k-k(k-1)-h=2k-h.\]
  But we also have
  \[r_{k}=\sum_{i=1}^kr_i-\sum_{i=1}^{k-1}r_i\le k(k-1)+h-(k-1)(k-2)=2k-2+h.\]
  Therefore
  \[(2k-h)+p\le r_{k+1}+p=r_k\le 2k-2+h,\]
  implying that $h\ge (p/2)+1$, a contradiction since $p\le3$.
  \end{proof}

\begin{theorem}\label{th:skew}
Let $n$ be an even integer, and let $R=(r_1,r_2,\ldots,r_{n/2},r_{n/2},\ldots,r_2,r_1)$ be a vector of nonnegative integers such that $(r_1,r_2,\ldots, r_{n/2})$ is nondecreasing. Then there exists a combinatorially skew-Hankel tournament with score vector $R$ if and only if
$(\ref{eq:self2})$ holds.
\end{theorem}

\begin{proof} We have already verified that (\ref{eq:self2}) holds for a combinatorially skew-Hankel tournament with score vector $R$. For the  converse, assume that (\ref{eq:self2}) holds but there does not exist a combinatorially skew-Hankel tournament with score vector $R$. We let $R$ be a counterexample with $n$ minimum and for this $n$ with $r_1$ minimum.

 As above, we let $R^*$ be the nondecreasing vector $(r_1,r_1,r_2,r_2,\ldots, r_{n/2},r_{n/2})$.
By Lemma \ref{lem:skew2}, $R^*$ satisfies (\ref{eq:self}).
By Lemma \ref{lem:skew1} if equality holds in (\ref{eq:self}) for an odd integer, it also holds for the next even integer. We consider two cases depending on whether or not equality holds in (\ref{eq:self}) for some even integer $l=2h<n$.

\smallskip\noindent
Case 1: There exists an integer $l=2h<n$ such that equality holds in (\ref{eq:self}).

Let $R'=(r_1,\ldots,r_h,r_h,\ldots,r_1)$. Since $r_{h+1}\ge 2h$, the vector
\[R''=(r_1'',\ldots,r_{(n/2)-h}'',r_{(n/2)-h}'',\ldots,r_1'')=(r_{h+1}-2h,\ldots,r_{n/2}-2h,r_{n/2}-2h,\ldots,r_{h+1}-2h)\]
is nonnegative.
Then $R'$ satisfies the inequality in (\ref{eq:self2}) for $1\le k\le h$ with equality if $k=h$. Also
\[\sum_{i=1}^k (r_i''+2h)=\sum_{i=1}^{h+k} r_i-\sum_{i=1}^h r_i\ge (h+k)(h+k-1) -h(h-1)=k(k-1)+2hk,\]
with equality when $k=(n/2)-h$. Thus
\[\sum_{i=1}^k r_i''\ge k(k-1)\mbox{ with equality if $k=(n/2)-h$,}\]
and $R''$ satisfies the corresponding condition (\ref{eq:self2}).
By our minimality assumption on $n$, there exist combinatorially skew-Hankel tournaments $T'$ and $T''$ with score vectors $R'$ and $R''$, respectively. We can write
\[T'=\left[\begin{array}{c|c}
T'_1&T'_2\\ \hline
(T'_2)^{th}& (T'_1)^{th}\end{array}\right],\]
where $T_1'$ and $T_2'$ are $h\times h$.
Then
\[T=\left[\begin{array}{c|c|c}
T'_1&O_{h,n-2h}&T'_2\\ \hline
J_{n-2h,h}&T'' &J_{n-2h,h}\\ \hline
(T'_2)^{th}&O_{h,n-2h}&(T'_1)^{th}\end{array}\right]\]
is a combinatorially skew-Hankel tournament with score vector $R$, a contradiction.

\smallskip\noindent
Case 2: There does not exists an integer $l<n$ such that equality holds in (\ref{eq:self}).

We first claim that the difference between the left hand side of (\ref{eq:self}) and its right hand side is at least 2 for $l<n$. If $l$ is even both sides are even so this holds.
Now suppose that $l=2k+1$ and, to the contrary, that
\[\sum_{i=1}^{2k+1}r_i^*=\left({{2k+1}\choose 2}-k\right) +1.\]
We calculate that
\[r_{2k+1}^*=\sum_{i=1}^{2k+1}r_i^*-\sum_{i=1}^{2k}r_i^*\le \left({{2k+1}\choose 2}-k\right) +1 -\left(
{{2k}\choose 2}-k+2\right)=2k-1,\]
and
\[r_{2k+2}^*=\sum_{i=1}^{2k+2}r_i^*-\sum_{i=1}^{2k+1}r_i^*\ge \left({{2k+2}\choose 2}-(k+1)+2 \right)  -\left(
{{2k+1}\choose 2}-k+1\right)=2k+1.\]
Thus
\[2k+1\le r_{2k+2}^*=r_{k+1}=r_{2k+1}^*\le 2k-1,\] a contradiction.
This verifies our claim.
 Now we consider the palindromic $R'=(r_1',\ldots,r_{n/2}',r_{n/2}',\ldots,r_1')$ where
 $r_1'=r_1-1$, $r_{n/2}'=r_{n/2}+1$, and $r_i'=r_i$ for $2\le i\le (n/2)-1$.
 Then the first half of $R'$ is nondecreasing. Since equality does not occur in (\ref{eq:self})
 for $l<n$, we have $r_1\ge 1$, and so $r_1'\ge 0$, and for $l\ge 2$, the conditions corresponding to (\ref{eq:self}) hold for $R'$. By the minimality assumption on $r_1$, there is a combinatorially skew-Hankel tournament $T'=[t_{ij}']$ with score vector $R'$. Since $r_{n/2}'-r_1'\ge 2$, there is a  column $p\ne 1$ such that $t_{1p}'=0$ and $t_{n/2,p}'=1$. Since $T'$ is a combinatorially skew-Hankel tournament, we have $p\ne n/2,(n/2)+1$. Assuming without loss of generality that $p<n/2$, the $6\times 6$ principal submatrix $T'[1,p,n/2,n/2+1,n+1-p,n]$ has the form
 \[\begin{array}{c||c|c|c||c|c|c}
 &1&p&n/2&n/2+1&n+1-p&n\\ \hline\hline
 1&0&\cellcolor[gray]{0.8}0&&&&0\\ \hline
 p&\cellcolor[gray]{0.8}1&0&\cellcolor[gray]{0.8}0&&0&\\ \hline
 n/2&&\cellcolor[gray]{0.8}1&0&0&&\\ \hline\hline
 n/2+1&&&0&0&\cellcolor[gray]{0.8}1&\\ \hline
 n+1-p&&0&&\cellcolor[gray]{0.8}0&0&\cellcolor[gray]{0.8}1\\ \hline
 n&0&&&&\cellcolor[gray]{0.8}0&0\end{array}.\]
 If we now change 0s to 1s and 1s to 0s in the shaded cells in the positions of $T'$, we obtain a combinatorially skew-Hankel tournament with score vector $R$, a contradiction completing the proof of the theorem.
\end{proof}

In case $n$ is odd we have the following characterization of score vectors of $n\times n$ combinatorially skew-Hankel tournaments.

\begin{theorem}\label{th:skew2}
Let $n$ be an odd integer, and let $R=(r_1,\ldots,r_{(n-1)/2},r_{(n+1)/2},r_{(n-1)/2},\ldots,r_1)$ be a vector of nonnegative integers such that $(r_1,r_2,\ldots, r_{(n-1)/2})$ is nondecreasing. Then there exists a combinatorially skew-Hankel tournament with score vector $R$ if and only if
\begin{equation}\label{eq:skew2}
\sum_{i=1}^kr_i\ge k(k-1) +\left(k-\frac{r_{(n+1)/2}}{2}\right)^+,\  \left(k=1,2,\ldots,\frac{n-1}{2}\right), \mbox{ with equality if $k=\frac{n-1}{2}$}.
\end{equation}
\end{theorem}

\begin{proof} The inequalities (\ref{eq:skew2}) are  necessary  for the existence of a $T\in{\mathcal T}_{H^*}(R)$. This is because in the $2k$ rows of $T$ with the smallest row sums, the minimum number of 1s that could be in column $(n+1)/2$ is $(2k-r_{(n+1)/2})^+$.

If we take an $n\times n$ combinatorially skew-Hankel tournament $T$ with score vector $R$ and delete the middle row and middle column, we are left with an $(n-1)\times (n-1) $ combinatorially 
skew-Hankel tournament with a nearly nondecreasing score vector. 
The conditions (\ref{eq:skew2})
imply that if we choose the middle column to have its 1s in the rows with the largest row sums (thereby determining the middle column) and delete the middle row and middle column, the row sum vector of the resulting matrix is a nearly nondecreasing vector $R'$ satisfying the conditions corresponding to (\ref{eq:self2}); by Lemma \ref{lem:skew3},  the nondecreasing rearrangement $R''$ of $R'$ also satisfies the same conditions. Hence by Theorem \ref{th:skew} there exists a combinatorially skew-Hankel tournament with score vector $R''$ and thus one with
$R'$. Hence there is a combinatorially skew-Hankel  tournament with score vector $R$.
\end{proof}

We now describe an algorithm to produce a combinatorially skew-Hankel tournament with a  prescribed score vector $R$ when the conditions of Theorems \ref{th:skew} and \ref{th:skew2} hold.

\medskip
\centerline{\bf Skew-Hankel Algorithm for a $T=[t_{ij}]\in {\mathcal T}_{H^*}(R)$ with $R$ nondecreasing}
\medskip

Let $R_n=(r_1,r_2,\ldots,r_n)$ be a nondecreasing  vector of nonnegative integers such that $r_i=r_{n+1-i}$ for all $i$. If $n$ is even, assume that (\ref{eq:self2}) holds. If $n$ is odd, assume that 
(\ref{eq:skew2}) holds.

\begin{enumerate}
\item If $n$ is odd:
\begin{enumerate}
\item[(a)] Let $v=(v_1,v_2,\ldots,v_n)$ be a $(0,1)$-vector with $v_i=v_{n+1-i}=1$ if and only if $
\frac{r_{(n+1)/2}}{2}+1\le i\le \frac{n-1}{2}$, and let the nearly nondecreasing vector 
 $R_{n-1}'=(r_1',r_2',\ldots,r_{n-1}')$ be defined by $
r_i'=r_{(n-1)+1-i}'=r_i-v_i$ for $1\le i\le (n-1)/2$.
\item[(b)] Let $Q_{n-1}$ be an $(n-1)\times (n-1)$ permutation matrix of the form
\[\left[\begin{array}{cc}
P&O_{(n-1)/2,(n-1)/2}\\
O_{(n-1)/2,(n-1)/2}&P^{th}\end{array}\right]\]
such that the first half of $R_{n-1}=R_{n-1}'Q_{n-1}^t$ is  nondecreasing. 
\item[(c)] Let $T_{n-1}$ be a combinatorially skew-Hankel tournament with score vector $R_{n-1}$ obtained by applying this algorithm, and let $T_{n-1}'$ be the combinatorially skew-Hankel tournament  with score vector $R_{n-1}'$ written in the form
\[T_{n-1}'=Q_{n-1}^t T_{n-1}Q_{n-1}=\left[\begin{array}{c|c}
A&B\\ \hline
B^{th}&A^{th}\end{array}\right].\]
\item[(d)] Let the combinatorially skew-Hankel matrix $T_n$ be defined by
\[T_n=\left[\begin{array}{ccc|c|ccc}
&&&v_1&&&\\
&A&&\vdots&&B&\\
&&&v_{(n-1)/2}&&&\\ \hline
(1-v_1)&\cdots&(1-v_{(n-1)/2})&0&(1-v_{(n+3)/2})&\cdots& (1-v_n)\\ \hline
&&&v_{(n+3)/2}&&&\\
&B^{th}&&\vdots&&A^{th}&\\
&&&v_n&&&\end{array}\right].\]
\end{enumerate}
\item If $n$ is even:
\begin{enumerate}
\item[(a)]  Let $v=(v_1,v_2,\ldots,v_n)$ be a $(0,1)$-vector with $v_i=1$ if and only if $
2+\lceil\frac{r_{1}}{2}\rceil\le i\le n-1-\lfloor\frac{r_1}{2}\rfloor$,
and let the 2-nearly nondecreasing vector 
 $R_{n-2}'=(r_1',r_2',\ldots,r_{n-2}')$ be defined by $
r_i'=r_{i+1}-v_{i+1}-v_{n-i}$ for $1\le i\le n-2$.
\item[(b)] Let $Q_{n-2}$ be an $(n-2)\times (n-2)$ permutation matrix such that the first half of  $R_{n-2}=R_{n-2}'Q_{n-2}^t$ is  nondecreasing and satisfies the skew-Hankel property.
\item[(c)] Let $T_{n-2}$ be a combinatorially skew-Hankel tournament with score vector $R_{n-2}$ obtained by applying this algorithm, and let $T_n$ be the combinatorially skew-Hankel tournament  with score vector $R_n$  defined by

\[T_n=\left[\begin{array}{c|c|c}
0&(1-v_{n-1})\cdots(1-v_2)&0\\ \hline
v_{n-1}&&v_2\\
\vdots&Q_{n-2}^tT_{n-2}Q_{n-2}&\vdots\\ 
v_2&&v_{n-1}\\ \hline
0&(1-v_2) \cdots (1-v_{n-1})&0\end{array}\right].\]
\end{enumerate}
(3) Output $T_n$.
\end{enumerate}

\begin{example}{\rm Let $n=7$ and let $R=(2,2,4,2,4,2,2)$. Applying the skew-Hankel algorithm, we obtain the following combinatorially skew-Hankel tournament in ${\mathcal T}_H(R)$:
\[\left[\begin{array}{c|c|c||c||c|c|c}
0&1&0&0&1&0&0\\  \hline
0&0&0&1&0&0&1\\ \hline
1&1&0&1&0&1&0\\ \hline\hline
1&0&0&0&0&0&1\\ \hline\hline
0&1&0&1&0&1&1\\ \hline
1&0&0&1&0&0&0\\ \hline
0&0&1&0&0&1&0\end{array}\right].\]
In carrying out the algorithm, the resulting score vectors are illustrated below where $\rightarrow_{\pi}$ means a permutation is used  and $\rightarrow_a$ means the result of a step of the algorithm:
\[(2,2,4,2,4,2,2)\rightarrow_a (2,1,3,3,1,2)\rightarrow_{\pi} (1,2,3,3,2,1)\rightarrow_a
(1,1,1,1) \rightarrow_a (0,0).\]
}\end{example}

\begin{theorem}\label{th:skewalg}
The skew-Hankel algorithm constructs a combinatorially skew-Hankel tournament 
when $R$ satisfies the given conditions.
\end{theorem}

\begin{proof} First assume that $n$ is even. It then suffices to verify that $R_{n-2}'$ 
is 2-nearly nondecreasing  and that $\sum_{i=1}^kr_i'\ge k(k-1)$
for $1\le k\le (n/2)-1$ (we have equality for $k=(n/2)-1$ by construction).
By construction
\[r_i'=r_{(n-2)+1-i}'=
\left\{\begin{array}{cl}
r_{i+1},&\mbox{ if $1\le i\le \lfloor r_1/2\rfloor$,}\\
r_{i+1}-1,&\mbox{ if $\lfloor r_1/2\rfloor+1\le i\le \lceil r_1/2\rceil$,}\\
r_{i+1}-2,&\mbox{ if $\lceil r_1/2\rceil+1\le i\le (n/2)-1$.}
\end{array}\right.\]
The entries of $R_{n-2}'$ are nonnegative, since $r_2=1$ (and so $r_1=1$) implies that $v_2=0$ and thus $r_1'\ge 0$.  We calculate that

\noindent
Case $1\le k\le \lfloor r_1/2\rfloor$: 
\[\quad \sum_{i=1}^kr_i'=\sum_{i=2}^{k+1} r_i\ge \sum_{i=1}^kr_i\ge k(k-1).\]
\noindent
Case $k=(r_1+1)/2$  and  $r_1$ odd, implying that $r_{k+1}\ge r_1=2k-1$:
\[\sum_{i=1}^kr_i'=\sum_{i=1}^{k-1}+r_{k+1}-1\ge (k-1)(k-2)+(2k-1)-1=k(k-1).\]
\noindent
Case $\lceil r_1/2\rceil+1\le k\le n/2-1$ \mbox{ with $r_1$ even}:
\begin{eqnarray*}
\sum_{i=1}^kr_i'&=&\left(\sum_{i=1}^{r_1/2+1}r_i\right)+\left(\sum_{i=r_1/2+2}^{k+1}r_i\right)-2\left((k+1)-((r_1/2)+2)+1\right)\\
&=&\left(\sum_{i=2}^{k+1}r_i\right)-2k+r_1=\left(\sum_{i=1}^{k+1}r_i\right)-2k\ge (k+1)k-2k=k(k-1).\\
\end{eqnarray*}
\noindent
Case $\lceil r_1/2\rceil+1\le k\le n/2-1$ \mbox{ with $r_1$ odd}:
\begin{eqnarray*}
\sum_{i=1}^k r_i'&=&\left(\sum_{i=2}^{k+1}r_i\right)-1-2\left((k+1)-((r_1+1)/2 +2)+1\right)\\
&=& \left(\sum_{i=1}^{k+1}r_i\right)-1-2k+1\ge (k+1)k-2k=k(k-1).\\
\end{eqnarray*}
Since $R$ is  nondecreasing, it follows that $R_{n-2}'$ is 2-nearly nondecreasing. By Lemma \ref{lem:skew3} the nondecreasing vector $R_{n-2}$ satisfies the corresponding inequalities. Therefore there is a combinatorially skew-Hankel tournament with score vector $R_{n-2}$ and also one with score vector $R_{n-2}'$.

Now assume that $n$ is odd.   The entries of $R_{n-1}'$  are
\[r_i'=r_{n-i}'=\left\{\begin{array}{cl}
r_i,&\mbox{ if  $1\le i\le (r_{(n+1)/2})/2$,}\\
r_i-1,& \mbox{ if $ (r_{(n+1)/2})/2+1\le i\le (n-1)/2$,}\end{array}\right.\]
and thus $R_{n-1}'$ is nearly nondecreasing. 
If $1\le k\le (r_{(n+1)/2})/2$, then $\sum_{i=1}^k r_i'\ge k(k-1)$. If $ (r_{(n+1)/2})/2+1\le k \le (n-1)/2$, then
\[\sum_{i=1}^kr_i'=\sum_{i=1}^k r_i-(k-(r_{(n+1)/2})/2)\ge k(k-1).\]
Therefore, there exists a combinatorially skew-Hankel tournament with nondeceasing score vector $R_{n-1}$ and also one with score vector $R_{n-1}'$.
\end{proof}

We next identify certain moves that allow one to move from any combinatorially skew-Hankel tournament $T\in{\mathcal T}_{H^*}(R)$  to any other where each move  produces another combinatorially skew-Hankel tournament in ${\mathcal T}_{H^*}(R)$. The switches used in the moves  are the 4-cycle switches and 3-cycle switches used in the case of Hankel tournaments, but their types are different. We consider 4-cycle switches
$\Box_{i,j,n+1-i,n+1-j}$ that reverse the 4-cycle $i\rightarrow j\rightarrow n+1-i\rightarrow n+1-j\rightarrow i$, 
called {\it $4$-cycle skew-Hankel switches},
and pure  3-cycle switches $\bigtriangleup_{i,j,k}$ that reverse $i\rightarrow j\rightarrow k\rightarrow i$, that is, 
where $\{i,j,k\}\cap\{n+1-i,n+1-j,n+1-k\}=\emptyset \mbox{ or }\{n/2\}$.
A 4-cycle skew-Hankel switch
$\Box_{i,j,n+1-i,n+1-j}$ is illustrated  by 
\[\begin{array}{c||c|c|c|c}
&i&j&n+1-i&n+1-j\\ \hline\hline
i&0&1&&0\\ \hline
j&0&0&1&\\ \hline
n+1-i&&0&0&1\\ \hline
n+1-j&1&&0&0\end{array}\rightarrow
\begin{array}{c||c|c|c|c}
&i&j&n+1-i&n+1-j\\ \hline\hline
i&0&0&&1\\ \hline
j&1&0&0&\\ \hline
n+1-i&&1&0&0\\ \hline
n+1-j&0&&1&0\end{array}.\]

First assume that $n$ is even.
Let $T_1,T_2\in {\mathcal T}_{H^*}$ and consider $T_1$ and $T_2$ as digraphs. For an $n\times n$  combinatorially skew-Hankel tournament, if there is an edge $i\rightarrow j$, then there is also an edge $n+1-i\rightarrow n+1-j$. Consider the digraph $D$  with vertex set $\{1,2,\ldots,n\}$  whose edges   are the edges of $T_1$ that are not edges of  $T_2$ (corresponding to the 1s in the difference matrix $T_1-T_2$).
If a vertex has positive outdegree then it has positive indegree, and vice versa. 
Starting at any vertex, determine a longest  path $\gamma$: $i\rightarrow j\rightarrow \cdots\rightarrow k$ that does not contain both a vertex $p$ and a vertex $n+1-p$. There is an arc leaving vertex $k$ and it must then go either to an earlier vertex on $\gamma$ or else to a vertex $n+1-p$ where $p$ is a vertex on $\gamma$.  Thus we either get  a cycle
\[({\rm a})\quad i_1\rightarrow i_2\rightarrow\cdots\rightarrow i_q\rightarrow i_1 \mbox{ where $i_j\ne n+1-i_l$
for any $ i_j$ and $i_l$,} \]
or a cycle
\[({\rm b})\quad i_1\rightarrow i_2\rightarrow \cdots\rightarrow i_q\rightarrow n+1-i_1\rightarrow n+1-i_2\rightarrow\cdots\rightarrow n+1-i_q\rightarrow i_1.\]
In case (a) we also have the cycle
\[({\rm a}') \quad n+1-i_1\rightarrow n+1-i_2\rightarrow\cdots\rightarrow n+1-i_q\rightarrow n+1-i_1.\]

\smallskip\noindent
Case (a):  We show how to reverse the cycle of (a) and of (a$'$) by a sequence of 3-cycle switches and their complementary switches.  
There is a sequence of pairs of switches, each consisting of a 3-cycle switch and its complementary switch, which   reverses the cycles (a) and (a$'$)  resulting in a  combinatorially skew-Hankel tournament  in ${\mathcal T}_{H^*}(R)$. We can find the first 3-cycle by considering all the edges of $T_1$ between $i_1$ and the other vertices of the cycle (a). Reversing this 3-cycle and its complementary 3-cycle, we obtain one or two shorter cycles. Repeating on these shorter cycles we eventually will reverse cycles (a) and (a$'$).

\smallskip\noindent
Case (b):
Consider, for instance, the edges $i_1\rightarrow i_2$ and $n+1-i_1\rightarrow n+1-i_2$. If the edge $n+1-i_2\rightarrow i_1$ is in $T_1$,  and thus so is the edge $i_2\rightarrow n+1-i_1$, then a 4-cycle skew-Hankel switch $\square_{i_1,i_2,n+1-i_1,n+1-i_2}$ yields a tournament in ${\mathcal T}_{H^*}$  where our cycle has become two cycles of the forms (a) and (a$'$), and we proceed inductively. If  the edge $i_1\rightarrow n+1-i_2$ is  in $T_1$, and thus so is the edge
$n+1-i_1\rightarrow i_2$, then we have two cycles of the forms (a) and (a$'$), which inductively by pairs of 3-cycle switches we can reverse, resulting in the reversal of $i_1\rightarrow n+1-i_2$ and $n+1-i_1\rightarrow i_2$.  Now a 4-cycle skew-Hankel switch  completes the reversal of the cycle of (b) and   gives a tournament in ${\mathcal T}_{H^*}(R)$ which is closer to $T_2$. 

For $n$ odd, we have the following lemma.

\begin{lemma}\label{lem:oddskew}
Let $T_1$ and $T_2$ be two combinatorially skew-Hankel matrices with the same score vector $R$. Then there exists  a sequence of moves, where each move is  a pair consisting of a pure $3$-cycle  switch
 $\bigtriangleup_{i,j,k}$ and its   complementary  switch $\bigtriangleup_{i,j,k}^c$, which transforms  row $($respectively, column$)$ $(n+1)/2$ of $T_1$ into row $($respectively, column$)$ $(n+1)/2$ of $T_2$.
\end{lemma}

\begin{proof}
As above we consider the digraph $D$. If row $(n+1)/2$ of $T_1$ agrees with row $(n+1)/2$ of $T_2$, there is nothing to prove. Otherwise there is a cycle $C$ containing vertex $(n+1)/2$.
We claim there is such a cycle  for which   if $i$ is a vertex of the cycle, then $(n+1-i)$ is not a vertex of the cycle.

Let $C$ be the cycle $(n+1)/2\rightarrow i_1\rightarrow i_2\rightarrow\cdots\rightarrow i_k\rightarrow (n+1)/2$. Let $V=\{j:i_j,(n+1-i_j)\in C, 1\le j\le k\}$. If $V\ne \emptyset$, let $a$ be the minimum and $c$ be the maximum of the integers in $V$.
Let $b$ and $d$ be the indices such that $i_b=n+1-i_a$ and $i_d=n+1-i_c$. We consider two cases (see Figure 2).

\smallskip\noindent
Case $b<d$: We consider the cycle $C'$ given by
\[\frac{n+1}{2}\rightarrow i_1\rightarrow\cdots\rightarrow
i_a\rightarrow (n+1-i_{b+1})\rightarrow\cdots\rightarrow
(n+1-i_{d-1})\rightarrow i_c\rightarrow \cdots\rightarrow i_k\rightarrow \frac{n+1}{2}.\]
It follows from the skew-Hankel property that the edges of $C'$ are in $D$, and for the corresponding $V'$, we have $|V'|<|V|$.

\smallskip\noindent
Case $b>d$: We then consider the cycle $C'$ given by
\[\frac{n+1}{2}\rightarrow i_1\rightarrow\cdots\rightarrow
i_a\rightarrow \cdots\rightarrow i_d\rightarrow (n+1-i_{c+1})\rightarrow\cdots\rightarrow
(n+1-i_{k})\rightarrow \frac{n+1}{2}.\]
(If $d=a$ and thus $c=b$, $i_a\rightarrow \cdots\rightarrow i_d$ reduces to just the vertex $i_a$.) As in the above case, we have $|V'|<|V|$. Repeating this process, we eventually produce a cycle $C^*$ such that $V^*=\emptyset$. This verifies our claim.

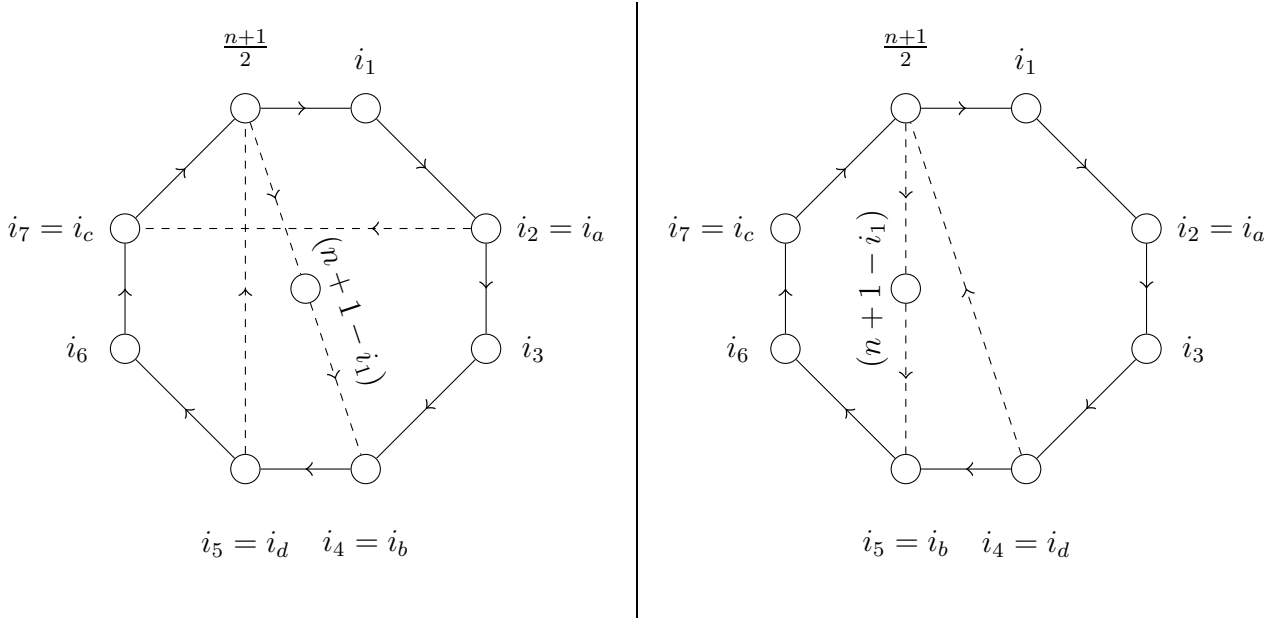
\begin{center}\begin{tabular}{c|c}
\begin{tikzpicture}
  [scale=0.8,auto=left,every node/.style={circle,draw,scale=1}]
  \node[label=above:$\frac{n+1}{2}$] (0) at (2,6) {};
  \node[label=above:$i_1$] (1) at (4,6) {};
  \node[label={right:$i_2=i_a$}] (2) at (6,4) {};
  \node[label=right:$i_3$] (3) at (6,2) {};
  \node[label={below:$i_4=i_b$}] (4) at (4,0) {};
  \node[label={below:$i_5=i_d$}] (5) at (2,0) {};
  \node[label=left:$i_6$] (6) at (0,2) {};
  \node[label={left:$i_7=i_c$}] (7) at (0,4) {};
  \node[label={[shift={(-1,-0.3)}]right:\rotatebox{-71.5}{$(n+1-i_1)$}}] (8) at (3,3) {};
  \draw[->-] (0) to (1);
  \draw[->-] (1) to (2);
  \draw[->-] (2) to (3);
  \draw[->-] (3) to (4);
  \draw[->-] (4) to (5);
  \draw[->-] (5) to (6);
  \draw[->-] (6) to (7);
  \draw[->-] (7) to (0);
  \draw[dashed,->-] (0) to (8);
  \draw[dashed,->-] (8) to (4);
  \draw[dashed,->-] (5) to (0);
  \draw[dashed,->--] (2) to (7);
\end{tikzpicture}
 &
\begin{tikzpicture}
  [scale=0.8,auto=left,every node/.style={circle,draw,scale=1}]
  \node[label=above:$\frac{n+1}{2}$] (0) at (2,6) {};
  \node[label=above:$i_1$] (1) at (4,6) {};
  \node[label={right:$i_2=i_a$}] (2) at (6,4) {};
  \node[label=right:$i_3$] (3) at (6,2) {};
  \node[label={below:$i_4=i_d$}] (4) at (4,0) {};
  \node[label={below:$i_5=i_b$}] (5) at (2,0) {};
  \node[label=left:$i_6$] (6) at (0,2) {};
  \node[label={left:$i_7=i_c$}] (7) at (0,4) {};
  \node[label={[shift={(1,0)}]left:\rotatebox{90}{$(n+1-i_1)$}}] (8) at (2,3) {};
  \draw[->-] (0) to (1);
  \draw[->-] (1) to (2);
  \draw[->-] (2) to (3);
  \draw[->-] (3) to (4);
  \draw[->-] (4) to (5);
  \draw[->-] (5) to (6);
  \draw[->-] (6) to (7);
  \draw[->-] (7) to (0);
  \draw[dashed,->-] (0) to (8);
  \draw[dashed,->-] (8) to (5);
  \draw[dashed,->-] (4) to (0);
\end{tikzpicture}
\end{tabular}\end{center}

\medskip

\centerline{Figure 2: Two possibilities in Lemma \ref{lem:oddskew} ($b<d$ and $b>d$)}

Now if row and column $(n+1)/2$ of $T_1$ do not agree with row and column $(n+1)/2$ of $T_2$,  we have a cycle $C$ 
\[(n+1)/2\rightarrow i_1\rightarrow i_2\rightarrow\cdots\rightarrow i_k\rightarrow (n+1)/2\] in $D$ (and so its reverse is in $T_2$) such that for each vertex $i$ of $C$, $(n+1-i)$ is not a vertex of $C$. By the skew-Hankel property, the cycle $C^{*}$ given by 
\[(n+1)/2\rightarrow (n+1- i_1)\rightarrow (n+1-i_2)\rightarrow\cdots\rightarrow (n+1- i_k\rightarrow (n+1)/2\] is also in  $D$ 
(and so its reverse is in $T_2$). Thus we can apply a sequence of moves, each of which is a 
pair consisting of a  3-cycle switch $\bigtriangleup_{i,j,k}$ and its complementary switch  $\bigtriangleup_{i,j,k}^c$,  to $T_1$ and  produce a $T_1'\in {\mathcal T}_{H^*}(R)$ 
such that   row and column $(n+1)/2$ of $T'$ are closer to
 row and column $(n+1)/2$ of $T_2$.
We continue like this until we get  row and column $(n+1)/2$ of $T_2$.
\end{proof}

Thus we have proved the following theorem.

\begin{theorem}\label{th:skewswitch}
Let $T_1$ and $T_2$ be two combinatorially skew-Hankel tournaments in ${\mathcal T}_{H^*}(R)$. Then there exists a sequence of $4$-cycle skew-Hankel switches $\Box_{i,j,n+1-i,n+1-j}$ and pairs consisting of a pure  $3$-cycle  switch 
$\bigtriangleup_{i,j,k}$ and  its complementary  switch $\bigtriangleup_{i,j,k}^c$ which brings $T_1$ to $T_2$, such that each of the moves produces  a combinatorially skew-Hankel tournament in 
 ${\mathcal T}_{H^*}(R)$.
\end{theorem}

\begin{example}{\rm 
Let $n=5$ and $R=(1,2,2,2,1)$, and consider the two combinatorially skew-Hankel tournaments $T_1$ and $T_2$ in ${\mathcal T}_{H^*}(R)$ given by
\[T_1=\left[
\begin{array}{c|c|c|c|c}
0&0&1&0&0\\ \hline
1&0&0&0&1\\ \hline
0&1&0&1&0\\ \hline
1&0&0&0&1\\ \hline
0&0&1&0&0\end{array}\right]\mbox{ and }
T_2=\left[
\begin{array}{c|c|c|c|c}
0&0&0&1&0\\ \hline
1&0&1&0&0\\ \hline
1&0&0&0&1\\ \hline
0&0&1&0&1\\ \hline
0&1&0&0&0\end{array}\right].\]
We can move from $T_1$ to $T_2$ by the  pair of pure 3-cycle switches $\bigtriangleup_{1,3,4}$ and $\bigtriangleup_{5,3,2}$.

}\end{example}

Again for completeness we mention the following. It is natural  to also consider $n\times n$  {\it combinatorially skew-Hankel loopy tournaments} $T=[t_{ij}]$, that is, combinatorially skew-Hankel tournaments
with possible $1$s on the main diagonal. By the skew-Hankel property, we now have that
$t_{ii}=1-t_{n+1-i,n+1-i}$ for all $i$.
The score vector $R=(r_1,r_2,\ldots,r_n)$ of a  combinatorially skew-Hankel loopy tournament satisfies
$r_i=r_{n+1-i}+1$ if $t_{ii}=1$ and $t_{n+1-i,n+1-i}=0$. 
 Thus the score vector of a  combinatorially skew-Hankel loopy tournament determines which elements on the main diagonal equal 1 and which equal 0. If $R'$ is obtained from $R$ by subtracting 1 from those $i$ for which $r_i=r_{n+1-i}+1$, then  there is a  combinatorially skew-Hankel loopy tournament with score vector $R$ if and only if there is a combinatorially  skew-Hankel tournament with score vector $R'$. 

By taking the columns of a combinatorially skew-Hankel loopy tournament in the reverse order, we obtain a {\it combinatorially skew-Hankel H-loopy  tournament}, that is, a combinatorially skew-Hankel tournament with possible 1s on the Hankel diagonal. Thus combinatorially skew-Hankel H-loopy  tournaments are equivalent to combinatorially skew-Hankel loopy tournaments.
But we may also consider {\it combinatorially skew-Hankel doubly-loopy tournaments}, that is, combinatorially skew-Hankel tournaments $T=[t_{ij}]$ with possible 1s on both the main diagonal and the Hankel diagonal. Thus $t_{ii}=1-t_{n+1-i,n+1-i}$ for all $i$ and $t_{i,n+1-i}=t_{n+1-i,i}$ for all $i$. 
The score vector $R=(r_1,r_2,\ldots,r_n)$ of a  combinatorially skew-Hankel doubly-loopy tournament satisfies
\[r_i=\left\{\begin{array}{cl}
r_{n+1-i},&\mbox{ if $t_{ii}=1$ and $t_{i,n+1-i}=0$, or $t_{n+1-i,i}=1$ and $t_{n+1-i,i}=0$}\\
r_{n+1-i}+2,&\mbox{ if $t_{ii}=t_{i,n+1-i}=1$ and so $t_{n+1-i,i}=t_{n+1-i,n+1-i}=0$.}
\end{array}\right.\] 
Let $R'$ be obtained from $R$ by subtracting 1 from $r_i$ and $r_{n+1-i}$  if $r_{i}=r_{n+1-i}$ and subtracting 2 from $r_i$ if $r_i=r_{n+1-i}+2$. Then $R$ is the score vector of a combinatorially skew-Hankel doubly-loopy tournament if and only if $R'$ is the score vector of a combinatorially skew-Hankel tournament.  Notice that when $r_i=r_{n+1-i}$ we can either put 1s in positions $(i,i)$ and $(n+1-i,i)$ or in positions
$(i,n+1-i)$ and $(n+1-i,n+1-i)$.

\section{Summary}

In this paper we have obtained necessary and sufficient conditions for the existence of loopy tournaments, Hankel tournaments, and combinatorially skew-Hankel tournaments with a prescribed score vector. We have also given algorithms for their construction 
when these conditions are satisfied. In addition, we have shown how to move from one tournament to another tournament in the same class by moves given by  switches and pairs of switches. The moves  used in each case are:
\begin{enumerate}
\item[(a)] loopy tournaments: $\hbox{$\rightarrow$}\kern -1.5pt\hbox{$\circ$}_{ij}$;  $\triangle_{i,j,k}$
\item[(b)] Hankel tournaments: $\triangle_{i,j,k}$  followed by $\triangle_{k,j,i}^c$; $\triangle_{i,(n+1)/2,n+1-i}$;  $\square_{i,j,n+1-j,n+1-i}$
\item[(c)] combinatorially skew-Hankel tournaments:  $\triangle_{i,j,k}$ followed by $\triangle_{i,j,k}^c$;  $\square_{i,j,n+1-i,n+1-j}$.
\end{enumerate}

\bibliographystyle{plain}

\end{document}